\documentclass{article}
\usepackage{pgfplots}
\usepackage[utf8]{inputenc}
\usepackage{amsfonts,amssymb,amsmath,amsthm}
\usepackage[a4paper, total={6in,9in}]{geometry}

\usepackage[utf8]{inputenc}
\usepackage{graphicx}
\usepackage{mathtools}

\usepackage{subcaption}
\usepackage{tikz}
\usepackage{pgfplots}
\usepackage{enumerate}
\usepackage{cases}
\usepackage{multicol}
\usepackage{sidecap}
\usepackage{float}
\usepackage{xcolor}
\definecolor{blue2}{rgb}{0.67, 0.9, 0.93}
\usepackage[color=blue2]{todonotes}
\usepackage{cases}
\usepackage{extarrows}
\usepackage{soul}
\usepackage[pagewise]{lineno}

\colorlet{siaminlinkcolor}{green!50!black}
\colorlet{siamexlinkcolor}{red!50!black}
\colorlet{siamreviewcolor}{black!50}

\usepackage[normalem]{ulem}
\usepackage[colorlinks=true,hyperfootnotes=false]{hyperref}
\hypersetup{
  colorlinks = true,
  allcolors = siaminlinkcolor,
  urlcolor = siamexlinkcolor,
}
\usepackage[nameinlink]{cleveref}
\usepackage{setspace}
\numberwithin{equation}{section}
\newtheorem{theorem}{Theorem}[section]
\newtheorem{lemma}[theorem]{Lemma}
\newtheorem{proposition}[theorem]{Proposition}
\newtheorem{cor}[theorem]{Corollary}
\newtheorem{hyp}{Hypothesis}
\newtheorem{remark}{Remark}

\newcommand{\N}{{\mathbb N}}
\newcommand{\R}{{\mathbb R}}

\newcommand{\Z}{\mathbb{Z}}

\renewcommand{\d}{{\mathrm d}}
\DeclareMathOperator{\sech}{sech}
\newcommand{\mand}{\quad \text{and}\quad}

\newcommand{\joinR}{\hspace{-.15em}}
\newcommand{\RomanV}{\scalebox{0.6}{\textit{V}}}
\newcommand{\RomanI}{\scalebox{0.6}{\textit{I}}}
\newcommand{\RomanII}{\mbox{\RomanI\joinR\RomanI}}
\newcommand{\RomanIII}{\mbox{\RomanI\joinR\RomanII}}
\newcommand{\RomanIV}{\mbox{\RomanI\joinR\RomanV}}

\usepackage{xspace}

\usepackage{authblk}
\title{\sc Radiating Solitary Waves in an FPUT Lattice with Random Coefficients }

\author[1]{H.J. Hupkes \thanks{\tt \href{mailto:hhupkes@math.leidenuniv.nl}{hhupkes@math.leidenuniv.nl}}}

\author[2]{J.A. McGinnis \thanks{\tt \href{mailto:jam887@sas.upenn.edu}{jam887@sas.upenn.edu}}}

\author[1]{R.W.S. Westdorp \thanks{\tt \href{mailto:r.w.s.westdorp@math.leidenuniv.nl}{r.w.s.westdorp@math.leidenuniv.nl}}}

\author[3]{J.D. Wright \thanks{\tt \href{mailto:jdw66@drexel.edu}{jdw66@drexel.edu}}}

\affil[1]{\small Mathematisch Instituut, Universiteit Leiden, P.O. Box 9512, 2300 RA Leiden, The Netherlands}

\affil[2]{\small Department of Mathematics, University of Pennsylvania, 209 South 33rd Street, Philadelphia, PA 19104 USA}

\affil[3]{\small Department of Mathematics, Drexel University, Korman Center, 33rd \& Market Streets, Philadelphia, PA 19104 USA}

\begin{document}

\maketitle
\begin{abstract}
We study the propagation of solitary waves in a Fermi–Pasta–Ulam–Tsingou (FPUT) lattice with small random heterogeneity in the linear spring force. Perturbed by the random environment, solitary waves lose energy through a radiative tail, resulting in gradual amplitude attenuation. As long as the wave remains coherent, we track its position and amplitude via a modulation approach. An expansion of the resulting modulation equations provides explicit predictions for the slow average amplitude decay, which we verify through numerical simulations.
\end{abstract}

\smallskip
\noindent\textbf{Keywords:}  FPUT; Lattice Differential Equations; Solitary waves; Random Coefficients

\smallskip
\noindent\textbf{MSC 2020:} 37K40; 37K60; 37H10

\section{Introduction}
When solitary waves in dispersive systems interact with spatial inhomogeneity, numerical simulations demonstrate that the waves emit a radiative tail \cite{akylas, beale, benilov, bona, simple}. This is decidedly so in Fermi–Pasta–Ulam–Tsingou (FPUT) lattices where the phenomenon has been observed in a variety of settings \cite{Pinski, Porter, Flach, giardetti, okada}. The study
of lattices with spatially varying material coefficients has a long history \cite{Maradudin, Dyson}, as disorder naturally arises in crystals
through mechanisms such as atomic replacement, isotopic variation, or structural
defects \cite{Bell}. It is well-known that homogeneous FPUT lattices support exact solitary wave solutions \cite{FPI, FPII}. In lattices with \textit{periodic} heterogeneity, propagation of localized solitary waves is obstructed by slow energy loss caused by oscillations in the tail \cite{Chirilus-Bruckner, Gaison, Mielke}. However, periodic lattices \textit{do} support 
the propagation of generalized solitary waves such as \textit{micropterons} and \textit{nanopterons} \cite{ nanopteron1, nanopteron2, nanopteron3, nanopteron4}, which consist of an exponentially localized core accompanied by a (typically) non-vanishing periodic tail. By contrast, lattices with randomly varying coefficients do not necessarily support such coherent propagation, and it is precisely this phenomenon that motivates the present study.

In this article, we explore the generation of radiating tails caused by random heterogeneity in FPUT lattices. Recent work 
\cite{ randomwalk,mcginnis} shows that long waves in such lattices will remain coherent over (somewhat) long time scales, depending on features of the randomness; this is accomplished by proving rigorous approximations by either wave or Korteweg-de Vries (KdV) equations. Nonetheless, over very long times the radiation has substantial effects on the principal solitary wave, causing (by virtue of energy conservation) a marked attenuation of the amplitude.  We use techniques devised in \cite{westdorp, westdorp2, westdorp3} to study random perturbations in KdV equations to analyze this attenuation.

\paragraph{FPUT}
We study the FPUT lattice equation
\begin{align}\label{eqn:FPU}
\dot{r}(j,t) =& \delta_+p(j,t), \\
\dot{p}(j,t) =& \delta_-[\mathcal{V}_\sigma' (r)](j,t),\qquad j\in\mathbb{Z},\nonumber
\end{align}
where
\[\delta_+f(j)=f(j+1)-f(j) \mand \delta_-f(j)=f(j)-f(j-1)\]
are the right and left finite-difference operators and 
\[\mathcal{V}_\sigma(r)=\frac{1}{2}(1+\sigma\kappa)r^2+\frac{1}{3}r^3\]
is the spring potential. The random sequence $\kappa$ models variations in the linear spring force, and the parameter $\sigma\geq 0$ controls the strength of this random effect. This terminology stems from the second-order form of the system
\begin{align}\ddot{y}(j,t)=&\mathcal{V}_\sigma' \big(y(j+1)-y(j)\big)-\mathcal{V}_\sigma' \big(y(j)-y(j-1)\big)\label{eqn:newton},\end{align}
 upon identifying $r(j)=y(j+1)-y(j)$ and ${p}(j)=\dot{y}(j)$. In this form, the system has the interpretation of a chain of masses connected by springs. The quantity $y(j)$ models the displacement of the $j$-th mass from equilibrium, and~\eqref{eqn:newton} is Newton's second law of motion.
 
 In the absence of heterogeneity $(\sigma=0)$, the lattice differential equation~\eqref{eqn:FPU} is also known as the FPUT-$\alpha$ lattice, referring to its cubic interaction potential. Upon introducing the Hamiltonian
\[
H_{\sigma\kappa}(r,p) = \sum_{j \in \Z} \frac{1}{2} p(j,t)^2 + \frac{1}{2}\big(1+\sigma\kappa(j)\big) r(j,t)^2 + \frac{1}{3} r(j,t)^3
\]
and the skew-symmetric operator
\[
\mathcal{J}:=\begin{bmatrix}
0 & \delta_+\\
\delta_- & 0
\end{bmatrix},
\]
we may alternatively write 
\begin{align}\label{eqn:hamiltonianform}
\begin{bmatrix}
    \dot{r}\\
    \dot{p}
\end{bmatrix} = \mathcal{J} H_{\sigma\kappa}'(r,p)=\mathcal{J}H_0'(r,p) + \sigma \mathcal{J}\begin{bmatrix}
    \kappa r\\
    0
\end{bmatrix}
.
\end{align}
Here, $H_{\sigma\kappa}'$ is the gradient of $H_{\sigma\kappa}$ with respect to the $\ell^2(\Z;\R^2)$ inner-product. Thus, the random coefficients we consider respect the Hamiltonian structure of the FPUT system. We assume the following:
\begin{hyp}\label{hyp:coeff}
The coefficients $\kappa(j)$, $j\in\Z$, are i.i.d.\footnote{We suspect that we can also handle more general translation-invariant covariance between the random coefficients, of the form 
\[\mathbb{E}[\kappa(i)\kappa(j)]=q(|i-j|),\]
assuming appropriate decay conditions on $(q(j))_{j\in\N}$.
For simplicity, we restrict ourselves in this work to i.i.d. coefficients.} random variables with mean zero and variance 1. The support of $\kappa(j)$, $j\in\Z$, is furthermore contained in an interval $[-\alpha,\alpha]$, and the parameter $\sigma\geq0$ satisfies $\sigma \alpha<1$. 
\end{hyp}
Probabilities associated with $\kappa$ are represented by $\mathbb{P}$ and expectations by $\mathbb{E}$. The local well-posedness of~\eqref{eqn:FPU} in $\ell^2(\Z;\R^2)$ is a straightforward consequence of the fact that $\mathcal{J}H_{\sigma\kappa}'$ is locally Lipschitz on $\ell^2(\Z;\R^2)$.  

\paragraph{Solitary waves}
 Solitary waves in FPUT lattices were extensively studied by Friesecke and Pego in the series \cite{FPI,FPII,FPIII,FPIV}. Notably, there exists a constant $c_+>1$, such that for all $c\in(1,c_+]$ there exists a smooth, exponentially decaying function $\phi_{c}= ( r_c,\ p_c)^\top: \R \to \R^2$ so that
\begin{align}
u(j,t) = \phi_c(j-ct) \label{eqn:exactwave}
\end{align}
solves~\eqref{eqn:FPU} with $\sigma=0$ \cite{FPI}. The profile function satisfies the advance-delay system:
\begin{align}
\label{eqn:TWE}
-c \phi_c' = \mathcal{J} H_0'(\phi_c).
\end{align}
The invariance of~\eqref{eqn:TWE} with respect to shifts in the profile coordinate $x$ implies that, for any $\xi \in \R$, \[\phi_{\xi,c} =\begin{bmatrix}r_{\xi,c}\\p_{\xi,c}\end{bmatrix}:=\begin{bmatrix}r_{c}(\cdot-\xi)\\p_{c}(\cdot-\xi)\end{bmatrix}\] also solves~\eqref{eqn:TWE}. Thus, the solitary waves form a two-parameter solution family, whose energy is independent of the phase $\xi$ and increases with the wave speed $c$ \cite[Theorem 1.1]{FPI}:
\[\frac{\d}{\d c}H_0(r_{c},p_{c})>0.\]
The FPUT lattice famously shares a strong connection with the KdV equation: in an appropriate continuum limit, the lattice dynamics are governed by an effective KdV equation \cite{zabusky, Schneider}. In this long-wave limit, the wave profiles $\phi_c$ approach the $\sech^2$ shape of KdV solitons \cite[Theorem 1.1]{FPI}:
\begin{align} 
 r_c(x) \approx -p_c(x) \approx  \frac{\epsilon^2}{8}\sech^2(\epsilon x/2),\label{eqn:approx}\end{align}
where the small parameter $\epsilon>0$ is defined through $c=1+\epsilon^2/24$. This gives the speed $c$ a second interpretation: $c-1$ is proportional to the (approximate) wave-amplitude.

\paragraph{Random Coefficients}
In \cite{mcginnis}, the second and fourth author derived effective KdV equations that govern the dynamics of an FPUT lattice with randomly varying mass coefficients. Central to this approximation is a \textit{transparency condition} on the random variation, enforcing that the random coefficients appear as a perfect discrete Laplacian. Numerical simulations reveal that, if the masses are instead perturbed by  i.i.d. random variables, a KdV approximation is no longer appropriate. Notably, solitary waves undergo an amplitude attenuation, inconsistent with KdV dynamics. In this work, we consider a random variation in spring coefficients rather than in the masses. This perturbation is somewhat more straightforward, as it arises purely at the linear level. We emphasize that our method can be applied more broadly, but we restrict ourselves here to spring force heterogeneity to keep the exposition transparent.

We study the effect of the random heterogeneity $\sigma\kappa$ in~\eqref{eqn:FPU} on the propagation of the solitary waves by supplying~\eqref{eqn:FPU} with the initial condition
\begin{align}\begin{bmatrix}r(j,0)\\p(j,0)\end{bmatrix}=\begin{bmatrix}r_{c_*}(j)\\p_{c_*}(j)\end{bmatrix},\quad j\in\Z,\label{eqn:initial}\end{align}
for some $c_*\in (1,c_{+})$. In the deterministic setting, the solitary wave family is known to be asymptotically stable, meaning that initial conditions close to the wave family asymptotically converge to it. Consequently, when the solitary wave 
is perturbed 
by the random coefficients $\kappa$, it does not fully disintegrate. Rather, it slowly loses energy over time and descends to lower speeds/amplitudes. In \Cref{fig:rovertime}, we can observe from a numerical simulation with i.i.d. coefficients (cf. \Cref{hyp:coeff}) that the loss of energy manifests in the formation of a random radiative `tail' behind the wave.

\begin{figure}[H]
    \centering
    \begin{subfigure}[t]{0.3\textwidth}
        \centering
        \includegraphics[width=\linewidth]{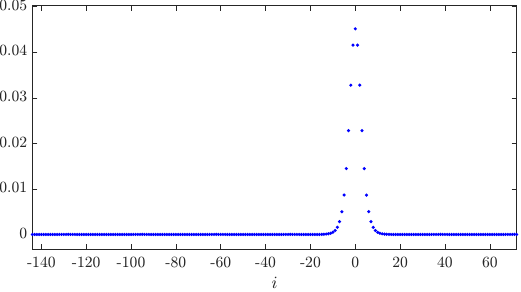}
        \caption{$t=0$.}
        \label{subfig:t=0}
    \end{subfigure}
    \hfill
    \begin{subfigure}[t]{0.3\textwidth}
        \centering
        \includegraphics[width=\linewidth]{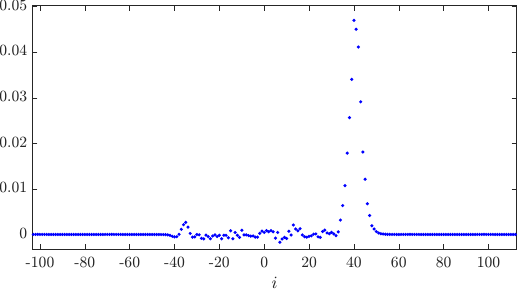}
        \caption{$t=40$.}
        \label{subfig:t=40}
    \end{subfigure}
    \hfill
    \begin{subfigure}[t]{0.3\textwidth}
        \centering
        \includegraphics[width=\linewidth]{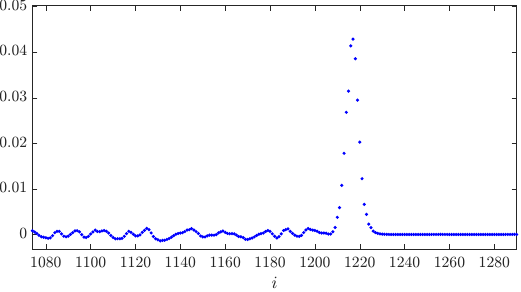}
        \caption{$t=1200$.}
        \label{subfig:t=1200}
    \end{subfigure}

    \caption{A particular realization of the $r$-component of numerical solutions to~\eqref{eqn:FPU} with initial condition~\eqref{eqn:initial}. In this realization, $\sigma=0.07$, $c_*=1.015$ and $\kappa(i)$ is drawn from a uniform distribution on $[-\sqrt{3},\sqrt{3}]$. See \Cref{app:scheme} for details regarding the numerical schemes used throughout this paper.}
    \label{fig:rovertime}
\end{figure}
\Cref{fig:sigmas} shows how the effective amplitude $c(t)$---which we define more precisely below---gradually decays with time as a consequence of the energy loss.  Although the change in amplitude is random, a clear attenuation emerges over time.
\begin{figure}[H]
    \centering
   \begin{subfigure}[t]{0.48\textwidth}
        \centering
        \includegraphics[width=\linewidth]{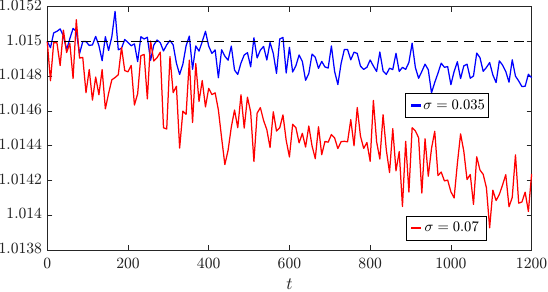}
\end{subfigure}
 \hfill
    \begin{subfigure}[t]{0.48\textwidth}
        \centering
        \includegraphics[width=\linewidth]{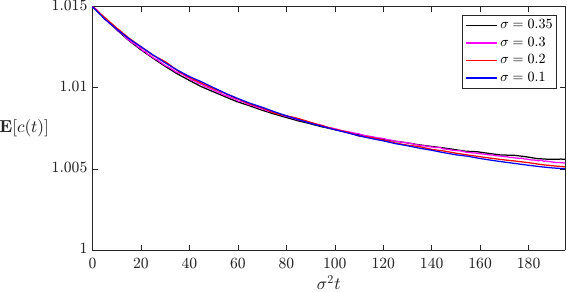}
    \end{subfigure}
        \caption{Left: Amplitude parameter $c(t)$ over time, for two realizations with $\sigma\in\{0.035,0.07\}$, $c_*=1.015$, and $\kappa(i)$ drawn from a uniform distribution on $[-\sqrt{3},\sqrt{3}]$. Right: Sample mean $\mathbf{E}[c(t)]$ computed over 2000 realizations, for $\sigma\in\{0.1,0.2,0.3,0.35\}$ and $c_*=1.015$. We emphasize the $\sigma^2$-scaling on the temporal axis.}
        \label{fig:sigmas}
   \end{figure} 

Our main objective in this work is to formally derive the attenuation induced by the random coefficients $\sigma\kappa$.
Using modulation theory, we obtain an explicit $\mathcal{O}(\sigma^2)$ prediction for the average amplitude decay, $\mathbb{E}[\dot{c}(t)]$, expressed in terms of integrals/sums involving the wave profiles $\phi_c$ and the Green's function associated to the linearized dynamics near $\phi_c$; see~\eqref{eqn:linearized} below. The latter plays a central role in the emission of a radiative tail behind the wave, which is the key mechanism driving amplitude attenuation. Our most explicit (and crudest) approximation is based on the kernel associated with the discrete wave equation, which can be written in closed form using Bessel functions. This leads to an explicitly solvable system that describes the radiation of the solitary wave and is able to reproduce the decay profile that emerges from \Cref{fig:sigmas} (right) in the limit $\sigma \downarrow 0$; see \Cref{fig:sigmastangents+ODE} ahead. 
This outcome is comparable in spirit to \cite{simple}, but now derived directly
from the physically relevant FPUT model. Our method is based on the linear stability theory for the wave profiles $\phi_c$, which we outline below.

\paragraph{Linear stability}

The linearized dynamics around a solitary wave $\phi_{c}(\cdot-\xi)$ are encoded by the operator $\mathcal{J}\mathcal{L}_{\xi,c}$, where $\mathcal{L}_{\xi,c}$ is the self-adjoint operator that acts on $\eta=(\eta_r,\ \eta_p)^\top\in\ell^2(\Z;\R^2)$ as
\begin{align}\mathcal{L}_{\xi,c}\eta=H_0''(\phi_{\xi,c})\eta=\begin{bmatrix}
        \eta_r + 2r_{c}(\cdot-\xi)\eta_r\\
        \eta_p 
    \end{bmatrix}.\label{eqn:LH}\end{align}
Indeed, linearizing~\eqref{eqn:FPU} with $\sigma = 0$ around the
traveling wave~\eqref{eqn:exactwave}, 
we arrive at the system
\begin{align}\dot{w}(j,t)=\mathcal{J}\mathcal{L}_{ct,c}w(j,t)=\mathcal{J}H_0''(\phi_{c}(\cdot-ct))w(j,t).\label{eqn:linearized}\end{align}
Note that the homogeneous part of $\mathcal{J}\mathcal{L}_{ct,c}$ is the discrete wave operator $\mathcal{J}$, whose continuous spectrum is contained in the imaginary axis. Differentiating~\eqref{eqn:TWE} with respect to $x$ and $c$ shows that
\[(\mathcal{J}\mathcal{L}_{\xi,c}+c\partial_x)\partial_\xi\phi_{\xi,c}=0,\mand (\mathcal{J}\mathcal{L}_{\xi,c}+c\partial_x)\partial_c\phi_{\xi,c}=\partial_\xi\phi_{\xi,c}.\]
Hence, $\partial_\xi\phi_{\xi,c}$ and $\partial_c\phi_{\xi,c}$ are generalized eigenfunctions of $\mathcal{J}\mathcal{L}_{\xi,c}+c\partial_x$, operating on functions of the real line. Here, $c\partial_x$ generates translations at the wave speed $c$, which are not directly represented in the linearization due to the discrete setting. Next, let $\delta_+^{-1}$ and $\delta_-^{-1}$ be given by
\[
\delta_+^{-1}f(j) = \sum_{d=-\infty}^{-1} f(j+d) \mand \delta_-^{-1}f(j)= \sum_{d = -\infty}^0 f(j+d),\quad f\in\ell^1(\Z;\R).
\]
These are the (formal) inverses of $\delta_\pm$. Then let
\[
\mathcal{J}^{-1} := \begin{bmatrix}
 0 & \delta_-^{-1} \\\delta_+^{-1} & 0    
\end{bmatrix} .
\]
This is the (formal) inverse of $\mathcal{J}$. The bilinear form
\[
\Omega(f,g) := \langle \mathcal{J}^{-1} f, g \rangle_{\ell^2(\Z;\R^2)}, \qquad f,g \in\ell^1(\Z;\R^2), 
\] 
is the \textit{symplectic form} for the FPUT system. This symplectic form characterizes a subspace that avoids the neutral modes $\partial_\xi\phi_{\xi,c}$ and $\partial_c\phi_{\xi,c}$, within which the linear flow defined in~\eqref{eqn:linearized} is exponentially stable.
\begin{proposition}[\cite{FPIII}, Theorem 1.2; \cite{FPIV}, Theorem 2.2]\label{prop:linearstability}
    Let $c_-\in(1,c_+)$. Then there exist positive constants $a,b$ and $K$ such that the following holds. For all $c\in [c_-,c_+ ]$ and $w_0$ satisfying $e^{a\cdot}w_0\in\ell^2(\Z;\R^2)$ with
    \begin{align}\Omega(\partial_\xi\phi_{c},w_0)=\Omega(\partial_c\phi_{c},w_0)=0,\label{eqn:orthostability}\end{align} the solution to the linearized evolution equation~\eqref{eqn:linearized} with $w(\cdot,0)=w_0$ admits the bound
    \[\|e^{a(\cdot-ct)}w(t)\|_{\ell^2(\Z;\R^2)}\leq Ke^{-bt}\|e^{a\cdot}w_0\|_{\ell^2(\Z;\R^2)}, \quad t\geq 0.\]
\end{proposition}
Henceforth, we fix $c_-\in (1,c_{+})$ and $a>0$ as above. In \cite{FPI}, the authors compute how the symplectic form $\Omega$ acts on the eigenfunctions. They report that:
\begin{align*}
 \Omega(\partial_\xi\phi_{\xi,c}, \partial_\xi\phi_{\xi,c}) =&0\\
\Omega(\partial_c\phi_{\xi,c},\partial_c\phi_{\xi,c}) =& \alpha_1(c)
\end{align*}
and 
\[ \Omega(\partial_\xi\phi_{\xi,c}, \partial_c\phi_{\xi,c}) = -\Omega(\partial_c\phi_{\xi,c}, \partial_\xi\phi_{\xi,c}) =\alpha_0(c),
\]
where
\begin{align}
\alpha_0(c) := {\frac{1}{c}} {\frac{\d}{\d c}} H(\phi_c)>0 \mand \alpha_1(c) := -\left( {\frac{\d} {\d c}} \int_\R r_c(x) \d x\right) \frac{\d}{\d c} \left(c \int_\R r_c(x) \d x \right).
\label{eqn:alpha0>0}  
\end{align}

\paragraph{Approach}
Motivated by the linear stability theory, we study the decomposition
\begin{align}\label{eqn:ansatz}
    \begin{bmatrix}
        r(j,t)\\
        p(j,t)
    \end{bmatrix}=\begin{bmatrix}
r_{c(t)}(j-\xi(t))\\
p_{c(t)}(j-\xi(t))
    \end{bmatrix}+\begin{bmatrix}
        \eta_r(j,t)\\
        \eta_p(j,t)
    \end{bmatrix},\quad j\in \Z
\end{align}
characterized by the orthogonality conditions
\begin{align}\label{eqn:ortho}\Omega\Big(\partial_\xi\phi_{c(t)}(\cdot-\xi(t)),\eta(t)\Big)=\Omega\Big(\partial_c\phi_{c(t)}(\cdot-\xi(t)),\eta(t)\Big)=0.
\end{align} 
Here, $\eta=(
    \eta_r,\ \eta_p)^\top$ denotes the deviation from the wave manifold 
    \[\mathcal{M}=\{\phi_c(\cdot-\xi)\ |\ \xi \in\R,\ c\in(c_-,c_+)\},\] 
    and avoids the neutral modes of the linearized flow. The result is that we can understand~\eqref{eqn:FPU} through the modulation coordinates $(\xi,c)$, where $\xi$ has the interpretation of wave position and $c$ of wave amplitude (see~\eqref{eqn:approx}). It is furthermore convenient to introduce a phase shift parameter $\gamma(t)$ through 
    \begin{align}\gamma(t)=\xi(t)-\int_0^t c(s) \ \d s.\label{eqn:phaseshift}\end{align}
    In these coordinates, the exact wave solution~\eqref{eqn:exactwave} for $\sigma=0$ takes the form 
    \[\big(\gamma(t),c(t),\eta(t)\big)=(0, c_*,0).\]
    For general $\sigma\geq 0$, the dynamics deviate from this exact solution and are governed by a coupled system of the form
\begin{align}
    \dot{\gamma}(t)=&\Gamma_{\sigma\kappa}(\xi(t),c(t),\eta(t)),\label{eqn:postl}\\
    \dot{c}(t)=&C_{\sigma\kappa}(\xi(t),c(t),\eta(t)),\label{eqn:postc}\\
    \dot{\eta}(t)=&\mathcal{J}\mathcal{L}_{\xi(t),c(t)}\eta(t)+T_{\sigma\kappa}(\xi(t),c(t),\eta(t)).\label{eqn:etadot} 
\end{align}
We study the resulting dynamics through a (formal) expansion of the system in the small parameter $\sigma$ around the deterministic solitary wave. In addition, we provide rigorous results to describe the asymptotic behavior of the expansion functions. We validate our findings through numerical simulations. In summary, we find that quadratic terms in the equation for $\dot{c}$---namely bilinear combinations of $\eta$ and $\kappa$---are the primary cause of the attenuation.

\paragraph{Outline} First, we derive the modulation equations that govern the evolution of $(\gamma,c,\eta)$ in \Cref{sec:modulation}, and present an expansion in $\sigma$ and $\eta$. Then, we introduce an expansion of the modulation system and resulting explicit approximations in \Cref{sec:reduced}. Next, we construct explicit solutions for linear approximations to the tail $\eta$ in \Cref{sec:radiative}. We then derive concrete predictions for the amplitude attenuation in \Cref{sec:attenuation}. Finally, we discuss directions for future research in \Cref{sec:outlook}.

\paragraph{Acknowledgments} Hupkes and Westdorp acknowledge support from the Dutch Research Council (NWO) (grant 613.009.137). Westdorp also thanks Drexel University for its hospitality during a research visit that contributed to this work. Wright acknowledges support from The Simons Foundation (MSP-TSM-00007725).
    
\section{Modulation system}\label{sec:modulation}

We first establish some fundamental properties of the decomposition~\eqref{eqn:ansatz} characterized by the orthogonality conditions~\eqref{eqn:ortho}. Our goal in this section is to derive the mappings $\Gamma_{\sigma\kappa}, C_{\sigma\kappa}$ and $T_{\sigma\kappa}$ in the modulation system~\eqref{eqn:postl}--\eqref{eqn:etadot}. We recall that the phase shift $\gamma(t)$ is related to the position $\xi(t)$ via~\eqref{eqn:phaseshift}, which implies that~\eqref{eqn:postl} is equivalent to
\begin{align}\dot{\xi}(t)=c(t)+\Gamma_{\sigma\kappa}(\xi(t),c(t),\eta(t)).\label{eqn:xidot}\end{align}
We will see that the mapping $T_{\sigma\kappa}$ in~\eqref{eqn:etadot} is given by
\begin{align*}
    T_{\sigma \kappa}(\xi,c,\eta)=\mathcal{J}N[\eta_r,\eta_r]-\Gamma_{\sigma \kappa}(\xi,c,\eta)
\partial_\xi\phi_{\xi,c}-C_{\sigma \kappa}(\xi,c,\eta)
\partial_c\phi_{\xi,c}+\sigma\begin{bmatrix}
     0\\
    \delta_-\big(\kappa (r_{\xi,c}+\eta_r)\big)
\end{bmatrix},
\end{align*}
    where $N[\eta_r,{\eta}_r]$ is the quadratic nonlinearity 
    \[N[\eta_r,\tilde{\eta}_r](j)=\begin{bmatrix}\eta_r(j) \tilde{\eta}_r(j)\\ 0\end{bmatrix}.\]
The mappings $\Gamma_{\sigma\kappa}$ and $C_{\sigma\kappa}$ are given by
\begin{align}
     \begin{bmatrix}
        \Gamma_{\sigma\kappa}(\xi,c,\eta)\\
        C_{\sigma\kappa}(\xi,c,\eta)
    \end{bmatrix}=&(A(c)-B(\xi,c)[\eta])^{-1}\left(\begin{bmatrix}
        \Omega(\partial_\xi\phi_{\xi,c},\mathcal{J}N[\eta_r,\eta_r])\\
        \Omega(\partial_c\phi_{\xi,c},\mathcal{J}N[\eta_r,\eta_r])
    \end{bmatrix}-\sigma\begin{bmatrix}
        \left\langle \partial_\xi r_{\xi,c},
     \kappa (r_{\xi,c}+\eta_r)\right\rangle_{\ell^2(\Z;\R)}\\
        \left\langle\partial_cr_{\xi,c},
     \kappa (r_{\xi,c}+\eta_r)\right\rangle_{\ell^2(\Z;\R)}
    \end{bmatrix}\right),\label{eqn:inverse}
\end{align}
    where $A(c)$ and $B(\xi,c)[\eta]$ are the matrices
\begin{align*}
    A(c)=\begin{bmatrix}
        \Omega(\partial_\xi\phi_c,\partial_\xi\phi_c) & \Omega(\partial_\xi\phi_c,\partial_c\phi_c) \\
        \Omega(\partial_c\phi_c,\partial_\xi\phi_c) & \Omega(\partial_c\phi_c,\partial_c\phi_c)
    \end{bmatrix}=\begin{bmatrix}
        0 & \alpha_0(c)\\
        -\alpha_0(c) & \alpha_1(c)
    \end{bmatrix},
\end{align*}
and
\begin{align}
    B(\xi,c)[\eta]=\begin{bmatrix}
        \Omega(\partial^2_{\xi\xi}\phi_{\xi,c},\eta) & \Omega(\partial^2_{c \xi}\phi_{\xi,c},\eta)\\
        \Omega(\partial^2_{\xi c}\phi_{\xi,c},\eta) & \Omega(\partial^2_{cc}\phi_{\xi,c},\eta)
    \end{bmatrix}.\label{eqn:bmatrix}
\end{align}
We remark that for all $c\in(c_-,c_+)$, the coefficient $\alpha_0(c)$ is strictly positive (see~\eqref{eqn:alpha0>0}), so that $A(c)$ is invertible with inverse
\begin{align}A^{-1}(c)=\frac{1}{\alpha_0^2(c)}\begin{bmatrix}
        \alpha_1(c) & -\alpha_0(c)\\
        \alpha_0(c) & 0
    \end{bmatrix}.\label{eqn:invertible}\end{align}
 Near the wave family $\mathcal{M}$, the unique existence and differentiability of our decomposition are guaranteed by the following result.
\begin{proposition}\label{prop:etamod}  
Assuming \Cref{hyp:coeff}, there exists a constant $\delta_*>0$ such that the following holds for all $T>0$. Let $u(t)=(r(t),p(t))$ be the unique solution to~\eqref{eqn:FPU} with initial condition~\eqref{eqn:initial} on $[0,T]$.  If $u(t)$ satisfies
    \[\inf_{\xi_*\in \R,c_*\in(c_-,c_+)}\|e^{a(\cdot-\xi_*)}\big(u(t)-\phi_{c_*}(\cdot-\xi_*)\big)\|_{\ell^2(\Z;\R^2)}\leq \delta_*, \quad t\in[0,T],\]
    then for each $t\in[0,T]$, there exist unique modulation parameters $(\xi(t),c(t),\eta(t))$ that satisfy~\eqref{eqn:ansatz} with~\eqref{eqn:ortho}. Furthermore, the map $t\mapsto (\xi(t),c(t),\eta(t))$ is differentiable on $[0,T]$ and satisfies the modulation equations~\eqref{eqn:postc},~\eqref{eqn:etadot} and~\eqref{eqn:xidot}.
\end{proposition}
\begin{proof}
The unique existence of the orthogonal decomposition~\eqref{eqn:ansatz} with~\eqref{eqn:ortho} near the wave manifold is proved in \cite[Proposition 2.2]{FPII}. We derive~\eqref{eqn:etadot} by differentiating~\eqref{eqn:ansatz} with respect to time:
\begin{align}
   \dot{\eta}(t)=\mathcal{J}H_0'(\phi_{\xi(t),c(t)}+\eta(t))-\dot{\xi}(t)
\partial_\xi\phi_{\xi(t),c(t)}-\dot{c}(t)
\partial_c\phi_{\xi(t),c(t)}+\sigma \mathcal{J} \begin{bmatrix}
     \kappa (r_{\xi(t),c(t)}+\eta_r(t))\\
    0
\end{bmatrix}
\label{eqn:doteta}
\end{align}
where
\begin{align*}\mathcal{J}H_0'(\phi_{\xi,c}+\eta)=&\mathcal{J}\begin{bmatrix}
        r_{\xi,c}+\eta_r + (r_{\xi,c}+\eta_r)^2\\
        p_{\xi,c}+\eta_p 
    \end{bmatrix}\\
    =&\mathcal{J}\begin{bmatrix}
        r_{\xi,c} + r_{\xi,c}^2\\
        p_{\xi,c}
    \end{bmatrix}+\mathcal{J}\begin{bmatrix}
        \eta_r + 2r_{\xi,c}\eta_r\\
        \eta_p 
    \end{bmatrix}+\mathcal{J}\begin{bmatrix}
        \eta_r^2\\0
    \end{bmatrix}\\
    =&c\partial_\xi\phi_{\xi,c}+\mathcal{J}\mathcal{L}_{\xi,c}\eta+\mathcal{J}N[\eta_r,\eta_r].
\end{align*}  
To derive~\eqref{eqn:postl} and~\eqref{eqn:postc}, we differentiate~\eqref{eqn:ortho}:
\begin{align}
    \begin{bmatrix}
        \Omega(\partial_\xi\phi_{\xi(t),c(t)},\dot{\eta}(t))\\
        \Omega(\partial_c\phi_{\xi(t),c(t)},\dot{\eta}(t))
    \end{bmatrix}+B(\xi(t),c(t))[\eta(t)]\begin{bmatrix}
        \Gamma_\kappa(\sigma,\xi(t),c(t),\eta(t))+c(t)\\
        C_{\kappa}(\sigma,\xi(t),c(t),\eta(t))
    \end{bmatrix}=0.\label{eqn:cancel}
\end{align}
We furthermore compute
\begin{align*}
     \begin{bmatrix}
        \Omega\Big(\partial_\xi\phi_{\xi(t),c(t)},\dot{\eta}(t)\Big)\\
        \Omega\Big(\partial_c\phi_{\xi(t),c(t)},\dot{\eta}(t)\Big)
    \end{bmatrix}\stackrel{\eqref{eqn:doteta}}{=}&\begin{bmatrix}
        \Omega\Big(\partial_\xi\phi_{\xi(t),c(t)},\mathcal{J}\mathcal{L}_{\xi(t),c(t)}\eta(t)+\mathcal{J}N[\eta_r(t),\eta_r(t)]\Big)\\
        \Omega\Big(\partial_c\phi_{\xi(t),c(t)},\mathcal{J}\mathcal{L}_{\xi(t),c(t)}\eta(t)+\mathcal{J}N[\eta_r(t),\eta_r(t)]\Big)
    \end{bmatrix}-A(c(t))\begin{bmatrix}
        \dot{\gamma}(t)\\
        \dot{c}(t)
    \end{bmatrix} \\
    &+\sigma \begin{bmatrix}
        \Omega\left(\partial_\xi\phi_{\xi(t),c(t)},\mathcal{J}\begin{bmatrix}
     \kappa (r_{\xi(t),c(t)}+\eta_r(t))\\
    0\end{bmatrix}\right)\\
        \Omega\left(\partial_c\phi_{\xi(t),c(t)},\mathcal{J}\begin{bmatrix}
     \kappa (r_{\xi(t),c(t)}+\eta_r(t))\\
    0\end{bmatrix}\right)
    \end{bmatrix}.
\end{align*}
The term linear in $\eta$ can be rewritten as
\begin{align*}
    \begin{bmatrix}
        \Omega(\partial_\xi\phi_{\xi,c},\mathcal{J}\mathcal{L}_{\xi,c}\eta)\\
        \Omega(\partial_c\phi_{\xi,c},\mathcal{J}\mathcal{L}_{\xi,c}\eta)
    \end{bmatrix}=&-\begin{bmatrix}
        \langle\partial_\xi\phi_{\xi,c},\mathcal{L}_{\xi,c}\eta\rangle_{\ell^2(\Z;\R^2)}\\
        \langle\partial_c\phi_{\xi,c},\mathcal{L}_{\xi,c}\eta\rangle_{\ell^2(\Z;\R^2)}   \end{bmatrix}\stackrel{\eqref{eqn:LH}}{=}-\begin{bmatrix}
            \langle \partial_\xi H_0'(\phi_{\xi,c}),\eta\rangle_{\ell^2(\Z;\R^2)}\\
            \langle \partial_c H_0'(\phi_{\xi,c}),\eta\rangle_{\ell^2(\Z;\R^2)}
        \end{bmatrix}\\=&-\begin{bmatrix}
            \Omega (\mathcal{J}\partial_\xi H_0'(\phi_{\xi,c}),\eta)\\
            \Omega (\mathcal{J}\partial_c H_0'(\phi_{\xi,c}),\eta)
        \end{bmatrix}
        \stackrel{\eqref{eqn:TWE}}{=}-c\begin{bmatrix}
        \Omega(\partial_{\xi\xi}^2\phi_{\xi,c},\eta)\\
            \Omega(\partial_{\xi c}^2\phi_{\xi,c},\eta)
        \end{bmatrix}=-c\begin{bmatrix}
            B_{11}(\xi,c,\eta)\\
            B_{21}(\xi,c,\eta)
        \end{bmatrix},
\end{align*}
and cancels with $B(\xi,c)[\eta](c,\ 0)^\top$ in~\eqref{eqn:cancel}.
The $\mathcal{O}(\sigma)$ term reduces to
\begin{align*}
    \begin{bmatrix}
        \Omega\left(\partial_\xi\phi_{\xi,c},\mathcal{J}\begin{bmatrix}
     \kappa (r_{\xi,c}+\eta_r)\\
    0\end{bmatrix}\right)\\
        \Omega\left(\partial_c\phi_{\xi,c},\mathcal{J}\begin{bmatrix}
     \kappa (r_{\xi,c}+\eta_r)\\
    0\end{bmatrix}\right)
    \end{bmatrix}=&-\begin{bmatrix}
        \Omega\left(\mathcal{J}\partial_\xi\phi_{\xi,c},\begin{bmatrix}
     \kappa (r_{\xi,c}+\eta_r)\\
    0\end{bmatrix}\right)\\
        \Omega\left(\mathcal{J}\partial_c\phi_{\xi,c},\begin{bmatrix}
     \kappa (r_{\xi,c}+\eta_r)\\
    0\end{bmatrix}\right)
    \end{bmatrix}\\
    =&-\begin{bmatrix}
        \left\langle \partial_\xi\phi_{\xi,c},\begin{bmatrix}
     \kappa (r_{\xi,c}+\eta_r)\\
    0\end{bmatrix}\right\rangle_{\ell^2(\Z;\R)}\\
        \left\langle\partial_c\phi_{\xi,c},\begin{bmatrix}
     \kappa (r_{\xi,c}+\eta_r)\\
    0\end{bmatrix}\right\rangle_{\ell^2(\Z;\R)}
    \end{bmatrix}\\
    =&-\begin{bmatrix}
        \left\langle \partial_\xi r_{\xi,c},
     \kappa (r_{\xi,c}+\eta_r)\right\rangle_{\ell^2(\Z;\R)}\\
        \left\langle\partial_cr_{\xi,c},
     \kappa (r_{\xi,c}+\eta_r)\right\rangle_{\ell^2(\Z;\R)}
    \end{bmatrix}.
\end{align*}
Finally, we note
\begin{align*}
\Big(A(c(t))-B(\xi(t),c(t))&[\eta(t)]\Big) \begin{bmatrix}
        \Gamma_\kappa(\sigma,\xi(t),c(t),\eta(t))\\
        C_\kappa(\sigma,\xi(t),c(t),\eta(t))
    \end{bmatrix}\\=&\begin{bmatrix}
        \Omega\Big(\partial_\xi\phi_{\xi(t),c(t)},\mathcal{J}N[\eta_r(t),\eta_r(t)]\Big)\\
        \Omega\Big(\partial_c\phi_{\xi(t),c(t)},\mathcal{J}N[\eta_r(t),\eta_r(t)]\Big)
    \end{bmatrix}\\
    &-\sigma\begin{bmatrix}
        \Big\langle \partial_\xi r_{\xi(t),c(t)},
     \kappa (r_{\xi(t),c(t)}+\eta_r(t))\Big\rangle_{\ell^2(\Z;\R)}\\
        \Big\langle\partial_cr_{\xi(t),c(t)},
     \kappa (r_{\xi(t),c(t)}+\eta_r(t))\Big\rangle_{\ell^2(\Z;\R)}
    \end{bmatrix}.
\end{align*}
Following~\eqref{eqn:invertible}, the invertibility of $A(c)-B(\xi,c)[\eta]$ is guaranteed for $c\in(c_-,c_+)$ and $\eta$ satisfying $\|e^{a\cdot}\eta\|_{\ell^2(\Z;\R^2)}\leq \delta_*$, upon decreasing $\delta_*$ if necessary. 
\end{proof}

\paragraph{Expansion}
The modulation system~\eqref{eqn:postl}--\eqref{eqn:etadot} describes how a small heterogeneity $\sigma>0$ affects the exact traveling wave $\big(\gamma(t),c(t),\eta(t)\big)=(0, c_*,0)$. It introduces $\mathcal{O}(\sigma)$ fluctuations in $c(t)$, and leads to an $\mathcal{O}(\sigma)$ phase shift $\gamma(t)$. The heterogeneity also develops an $\mathcal{O}(\sigma)$ tail $\eta(t)$, which we expect to remain small for a long time, due to the stabilizing effect of the linearized dynamics. In order to expose the structure of the coupled modulation system~\eqref{eqn:postl}--\eqref{eqn:etadot}, we expand the system in the parameters $\sigma$ and $\eta$:
 \begin{align}
\Gamma_{\sigma\kappa}(\xi,c,\eta)=&\sigma\big(\Gamma^{1,0}(\xi,c)[\kappa]+\Gamma^{1,1}(\xi,c)[\kappa,\eta]\Big)+\Gamma^{0,2}(\xi,c)[\eta_r,\eta_r]+\mathcal{O}(\sigma\eta^2)+\mathcal{O}(\eta^3),\nonumber\\
C_{\sigma\kappa}(\xi,c,\eta)=&\sigma\big(C^{1,0}(\xi,c)[\kappa]+C^{1,1}(\xi,c)[\kappa,\eta]\Big)+C^{0,2}(\xi,c)[\eta_r,\eta_r]+\mathcal{O}(\sigma\eta^2)+\mathcal{O}(\eta^3),\label{eqn:expansion}\\
T_{\sigma\kappa}(\xi,c,\eta)=&\sigma\big(T^{1,0}(\xi,c)[\kappa]+T^{1,1}(\xi,c)[\kappa,\eta]\Big)+T^{0,2}(\xi,c)[\eta_r,\eta_r]+\mathcal{O}(\sigma\eta^2)+\mathcal{O}(\eta^3).\nonumber
\end{align}
Here, the first index $k\in\{0,1\}$ in $\Gamma^{k,n}$, $C^{k,n}$ and $T^{k,n}$ refers to the accompanying factor $\sigma^k$ in~\eqref{eqn:expansion}, whereas  the second index $n\geq0$ indicates that the mapping is purely of order $\mathcal{O}(\eta^n)$. To 
compute these expansion terms,
we note that the matrix inverse in~\eqref{eqn:inverse} can be expanded through a Neumann series:
\begin{align*}
    \big(A(c)-B(\xi,c)[\eta]\big)^{-1}=&\big(I-A^{-1}(c)B(\xi,c)[\eta]\big)^{-1}A^{-1}(c)\\
    =&\sum_{k=0}^\infty\big(A^{-1}(c)B(\xi,c)[\eta]\big)^kA^{-1}(c).
\end{align*}
For each $\xi\in\R$ and $c\in(c_-,c_+)$, the maps $\Gamma^{1,0}(\xi,c)$ and $C^{1,0}(\xi,c)$ are linear operators from $\ell^\infty(\Z;\R)$ to $\R$ that act as
\begin{align}\begin{bmatrix}
         \Gamma^{1,0}(\xi,c)\\
         C^{1,0}(\xi,c)
     \end{bmatrix}[f]=&- \tfrac{1}{2}A^{-1}(c)\begin{bmatrix}
       \langle \partial_\xi r^2_{c}(\cdot-\xi),f\rangle_{\ell^2(\Z;\R)}\\
        \langle \partial_cr^2_{c}(\cdot-\xi),f\rangle_{\ell^2(\Z;\R)}
    \end{bmatrix}.\label{eqn:Ls}\end{align}
At order $\mathcal{O}(\eta^2)$, we find that $\Gamma^{0,2}(\xi,c)[\cdot,\cdot]$ and $C^{0,2}(\xi,c)[\cdot,\cdot]:$ are the bilinear maps from $\ell^\infty(\Z;\R)\times \ell^\infty(\Z;\R)$ to $\R$ that act as
\begin{align}
\begin{bmatrix}
         \Gamma^{0,2}(\xi,c)\\
         C^{0,2}(\xi,c)
     \end{bmatrix}[f,g]:=&-A^{-1}(c)\begin{bmatrix}
        \langle\partial_\xi r_{c}(\cdot-\xi),fg\rangle_{\ell^2(\Z;\R)}\\
        \langle \partial_cr_{c}(\cdot-\xi),fg\rangle_{\ell^2(\Z;\R)}
    \end{bmatrix}.
\label{eqn:bilinear1}\end{align}
In particular, 
$\Gamma^{0,2}(\xi,c)[\cdot,\cdot]$ and $C^{0,2}(\xi,c)[\cdot,\cdot]$ 
can be seen as weighted inner products when restricted to 
$\ell^2(\Z;\R) \times \ell^2(\Z;\R)$. Lastly, $\Gamma^{1,1}(\xi,c)[\cdot,\cdot]$ and $C^{1,1}(\xi,c)[\cdot,\cdot]$ are bilinear maps from $\ell^\infty(\Z;\R)\times\ell^\infty(\Z;\R^2)$ to $\R$ which act as 
\begin{align}\begin{bmatrix}
         \Gamma^{1,1}(\xi,c)\\
         C^{1,1}(\xi,c)
     \end{bmatrix}[f,h]=A^{-1}(c)B(\xi,c)[h]\begin{bmatrix}
         \Gamma^{1,0}(\xi,c)\\
         C^{1,0}(\xi,c)
     \end{bmatrix}[f]+\begin{bmatrix}
         \Gamma^{0,2}(\xi,c)\\
         C^{0,2}(\xi,c)
     \end{bmatrix}[f,h_r].\label{eqn:bilinear2}\end{align}
     The expansion of~\eqref{eqn:postl}--\eqref{eqn:etadot} up to second-order is completed by the linear map $T^{1,0}(\xi,c):\ell^\infty(\Z;\R)\to\ell^2(\Z;\R^2)$ that acts as
 \begin{align} T^{1,0}(\xi,c)[f]=& -\Gamma^{1,0}(\xi,c)[f]
\partial_\xi\phi_{\xi,c}-C^{1,0}(\xi,c)[f]
\partial_c\phi_{\xi,c}+ \begin{bmatrix}
     0\\
    \delta_-\big(f r_{\xi,c}\big)
\end{bmatrix},\label{eqn:Tk}
\end{align}
and the bilinear maps 
\[T^{1,1}(\xi,c)[\cdot,\cdot]:\ell^\infty(\Z;\R)\times\ell^\infty(\Z;\R^2)\to\ell^2(\Z;\R^2)\] and
\[T^{0,2}(\xi,c)[\cdot,\cdot]:\ell^\infty(\Z;\R)\times\ell^\infty(\Z;\R)\to\ell^2(\Z;\R^2)\]
that act as
\begin{align*}
 T^{1,1}(\xi,c)[f,h]=&-\Gamma^{1,1}(\xi,c)[f,h]
\partial_\xi\phi_{\xi,c}-C^{1,1}(\xi,c)[f,h]
\partial_c\phi_{\xi,c}+ \begin{bmatrix}
     0\\
    \delta_-\big(f h_r\big)
\end{bmatrix},\\
T^{0,2}(\xi,c)[f,g]=&
 -\Gamma^{0,2}(\xi,c)[f,g]
\partial_\xi\phi_{\xi,c}-C^{0,2}(\xi,c)[f,g]
\partial_c\phi_{\xi,c}+\mathcal{J}N[f,g].\end{align*}
Although higher-order contributions can be identified in the same manner, we will only make use of the $\mathcal{O}(\sigma)$, $\mathcal{O}(\sigma\eta)$ and $\mathcal{O}(\eta^2)$ contributions identified above. This is because they suffice for our purposes of studying the $\mathcal{O}(\sigma^2)$ amplitude attenuation of $c(t)$, given that $\eta$ develops at order $\mathcal{O}(\sigma)$. See also \Cref{fig:sigmas}. In the following section, we construct explicit approximations based on these second-order terms to facilitate further analysis.

\section{System expansion}\label{sec:reduced}
Having derived the modulation system~\eqref{eqn:postl}--\eqref{eqn:etadot}, we turn our attention to the resulting dynamics. We emphasize that the modulation parameters $(\gamma,c,\eta)$ are, through a coordinate transformation, equivalent to the original system $(r,p)$. Hence, solving for $c(t)$ should reveal the observed amplitude attenuation. However, due to its coupled nature, this information is hard to extract directly from~\eqref{eqn:postl}--\eqref{eqn:etadot}.

Inspired by \cite{westdorp}, this section introduces explicit approximations based on the expansion~\eqref{eqn:expansion}. The key issue is that the expansion coefficients $\Gamma^{k,n}$, $C^{k,n}$ and $T^{k,n}$ depend explicitly on the pair $(\xi, c)$. At each step in the expansion procedure one can choose to treat this pair as free functions that must be solved for, but it is also possible to insert an a-priori approximation. The approximations developed in \cite{westdorp} to describe soliton propagation in the KdV equation were based on the former approach. The main reason is that the random forcing that was considered broke the Hamiltonian structure of the system. In particular, 
fluctuations appearing at order $\mathcal{O}(\sigma)$ lead to significant deviations from the original amplitude on short time scales, requiring the use of an adaptive approach to account for the short-time fluctuations.

In this work however, the random coefficients in~\eqref{eqn:FPU} preserve the Hamiltonian structure of the FPUT system. As a consequence, the dynamics of $c(t)$ induced by terms at $\mathcal{O}(\sigma)$ do not accumulate significantly over time---see \Cref{prop:variance} ahead. In the present context, $c(t)$ remains close to its initial value $c_*$ over time scales proportional to $\sigma^{-2}$, leading to a separation of time scales. In particular, transient dynamics such as the development of the tail $\eta(t)$ occur at an exponential rate independent of $\sigma$---see \Cref{cor:limit} ahead. We therefore depart from an adaptive approach, and simply expand the modulation system~\eqref{eqn:postl}--\eqref{eqn:etadot} in the small parameter $\sigma$, around the deterministically  propagating soliton $\phi_{c_*}(\cdot-c_*t)$:
\begin{align}
    \gamma(t)=&\sigma\gamma_1(t)+\sigma^2\gamma_2(t)+\mathcal{O}(\sigma^3),\label{eqn:expgamma}\\
    c(t)=&c_*+\sigma c_1(t)+\sigma^2c_2(t)+\mathcal{O}(\sigma^3),\label{eqn:expc}\\
    \eta(t)=&\sigma \eta_1(t)+\mathcal{O}(\sigma^2).\label{eqn:expeta}
\end{align}
Since the phase shift $\gamma(t)$ is related to the position $\xi(t)$ through~\eqref{eqn:phaseshift}, we have an equivalent expansion
\[\xi(t)=c_*t+\sigma \xi_1(t)+\sigma^2\xi_2(t)+\mathcal{O}(\sigma^3),\]
with
\[\xi_i(t)=\gamma_i(t)+\int_0^tc_i(s) \ \d s, \qquad i=1,2.\]
In principle, the expansion can be continued up to any desired order in $\sigma$, but we will only  consider the explicit terms appearing in~\eqref{eqn:expgamma}--\eqref{eqn:expeta}. In \Cref{sec:attenuation} ahead, we  compute the expected attenuation of the amplitude approximation $c_*+\sigma c_1(t)+\sigma^2c_2(t)$. Although $c(t)$ deviates significantly from its initial value after long times, the expansion captures the attenuation occurring over intermediate timescales. In this sense, the attenuation computed from an expansion around $\phi_{c_*}(\cdot-c_*t)$ can be interpreted as representing tangents to the curves in \Cref{fig:sigmas} at $\mathbf{E}[c(t)]=c_*$. See \Cref{fig:sigmastangents+ODE} in \Cref{sec:attenuation} for further details.

\paragraph{leading-order phase and amplitude}
Collecting $\mathcal{O}(\sigma)$ terms in~\eqref{eqn:postl} and~\eqref{eqn:postc} yields 
\begin{align}\begin{bmatrix}
        \dot{\gamma}_{1}(t)\\
        \dot{c}_{1}(t)
    \end{bmatrix}=\begin{bmatrix}
        \Gamma^{1,0}(c_*t,c_*)\\
        C^{1,0}(c_*t,c_*)
    \end{bmatrix}[\kappa],\quad \text{with}\quad \begin{bmatrix}
        {\gamma}_{1}(0)\\
        {c}_{1}(0)
    \end{bmatrix}=\begin{bmatrix}
        0\\
        0
    \end{bmatrix}.\label{eqn:c1}\end{align}
These leading-order terms ignore $\eta$ altogether, since it forms an $\mathcal{O}(\eta^2)\sim\mathcal{O}(\sigma^2)$ contribution. We note that
\begin{align*}
        \begin{bmatrix}
        \Gamma^{1,0}(c_*t,c_*)\\
        C^{1,0}(c_*t,c_*)
    \end{bmatrix}[\kappa]=-\tfrac{1}{2}A^{-1}(c_*)\sum_{j\in\Z}\kappa(j)\begin{bmatrix}
        (\partial_\xi r^2_{c_*})(j-c_*t)\\
        (\partial_cr^2_{c_*})(j-c_*t)
    \end{bmatrix},
\end{align*}
through which we arrive at the explicit solution 
\begin{align}
c_{1}(t)=&  \tfrac{1}{2}\alpha_0^{-1}(c_*)\sum_{j\in\Z}\kappa(j)\int_0^t (r^2_{c_*})'(j-c_*s)\d s\nonumber \\
=& - \frac{\alpha_0^{-1}(c_*)}{2c_*}\sum_{j\in\Z}\kappa(j)[r^2_{c_*}(j-c_*t)-r^2_{c_*}(j)].\label{eqn:c1explicit}\end{align}
The bracketed terms in~\eqref{eqn:c1explicit} form (shifting) windows that select the lattice sites at which the random coefficients contribute. Similarly, we compute that $\gamma_{1}(t)= \gamma_{1;I}(t)+\gamma_{1;II}(t)$ with 
\begin{align}
\gamma_{1;\RomanI}(t)=& 
- \frac{\alpha^{-2}_0(c_*)\alpha_1(c_*)}{2c_*}\sum_{j\in\Z}\kappa(j)[r^2_{c_*}(j-c_*t)-r^2_{c_*}(j)]\label{eqn:Y1}\end{align}
 and
 \begin{align}
\gamma_{1;\RomanII}(t)=& \frac{1}{2\alpha_0(c_*)}\sum_{j\in\Z}\kappa(j)\int_0^t\partial_c r^2_{c_*}(j-c_*s)\d s.\label{eqn:Y2}\end{align}
We note that
\[\int_{-\infty}^x\partial_c r^2_{c_*}(y)\d y\to \int_\R \partial_c r^2_{c_*}(y)\d y>0,\quad \text{as }x \to \infty.\]
Thus, the integral in~\eqref{eqn:Y2} acts as a widening window, selecting an increasing amount of lattice sites with time. Hence, $\gamma_{1;II}$ resembles a continuous random walk, and its variance grows linearly with time. Below, we collect some elementary properties regarding the statistics of $c_1(t)$ and $\gamma_1(t)$.
\begin{proposition}[See \Cref{app:statistics}]\label{prop:variance}
\par\noindent
Assume \Cref{hyp:coeff} and let $c_*\in(c_-,c_+)$. Then
    \begin{enumerate}
        \item $c_{1}(t)$ is a mean-zero random process, whose variance is uniformly bounded in time as
         \begin{align}
         \mathbb{E}[c_{1}^2(t)]\leq  \frac{\alpha_0^{-2}(c_*)}{c_*^2}\|r_{c_*}\|_{\ell^4(\Z;\R)}^4,\quad t\geq 0.\label{eqn:variancebound}\end{align}
        \item $\gamma_{1}(t)$ is a mean-zero random process, whose variance is bounded linearly in time as
         \[\mathbb{E}[\gamma_{1}^2(t)]\leq \big(C_1 + C_2t\big),\quad t\geq0,\]
         where the constants $C_1, C_2>0$ depend only on $c_*$. 
    \end{enumerate}
\end{proposition}
\Cref{fig:zero order} shows simulations of the first-order parameters $(c_{1}(t), \gamma_{1}(t))$. The approximation $c_*+\sigma c_{1}(t)$ captures the leading-order effect of the random coefficients: the solitary wave experiences an $\mathcal{O}(\sigma)$ `shaking', but the disturbances do not accumulate with time. See the left panel of \Cref{fig:zero order}: The approximation $c_*+\sigma c_{1}(t)$ does not capture the attenuation of the soliton. 
\begin{figure}[htbp]
    \centering
    \begin{subfigure}[t]{0.48\textwidth}
        \centering
        \includegraphics[width=\linewidth]{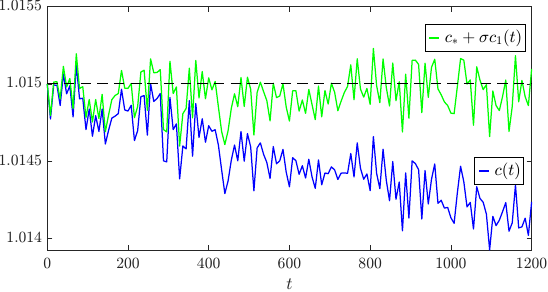}
    \end{subfigure}
    \hfill
    \begin{subfigure}[t]{0.48\textwidth}
        \centering
        \includegraphics[width=\linewidth]{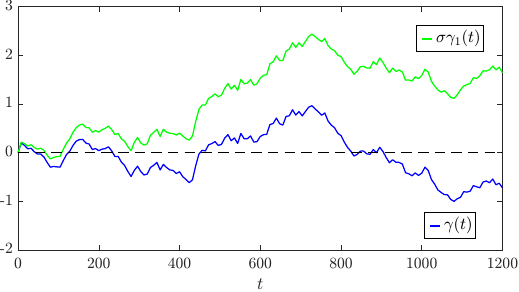}
    \end{subfigure}

    \caption{Comparison of first-order approximations $c_*+\sigma c_1(t)$ and $\sigma \gamma_1(t)$ of amplitude and position, for a particular realization with $\sigma=0.07$ and $c_*=1.015$. The random coefficients $\kappa(i)$ are drawn from a uniform distribution on $[-\sqrt{3},\sqrt{3}]$. The parameters ${c}_{1}(t)$ and $\gamma_{1}(t)$ correspond to a numerical simulation of the approximation~\eqref{eqn:c1}. These processes capture fluctuations, but do not capture the attenuation affecting $c(t)$.}
    \label{fig:zero order}
\end{figure}

\paragraph{Tail approximations}

Analogously, we identify the leading-order expansion term $\sigma \eta_{1}$ of $\eta$ by isolating $\mathcal{O}(\sigma)$ contributions in~\eqref{eqn:etadot}. In particular, we find 
\begin{align}\dot{\eta}_{1}(t)=&\mathcal{J}\mathcal{L}_{c_* t,c_*}\eta_{1}(t)+ T^{1,0}(c_* t,c_*)[\kappa]\label{eqn:etalin}\\
=& \mathcal{J}\mathcal{L}_{c_*t,c_*}\eta_{1}(t)+\begin{bmatrix}
      \partial_\xi\phi_{c_*t,c_*} \\
      \partial_c\phi_{c_*t,c_*}
    \end{bmatrix}^\top  A^{-1}(c_*)\begin{bmatrix}
        \left\langle \partial_\xi r_{c_*t,c_*},
     \kappa r_{c_*t,c_*}\right\rangle_{\ell^2(\Z;\R)}\\
        \left\langle\partial_cr_{c_*t,c_*},
     \kappa r_{c_*t,c_*}\right\rangle_{\ell^2(\Z;\R)}
    \end{bmatrix}
+\sigma \mathcal{J}\begin{bmatrix}
     \kappa r_{c_*t,c_*}\\
    0
\end{bmatrix},\nonumber
\end{align}
supplemented with the initial condition $\eta_{1}(0,j)=0$, $j\in\Z$. This is now a linear equation, forced by the random coefficients $\kappa$. We recall that $\mathcal{L}_{c_*t,c_*}$ is the linearization around the solitary wave $\phi_{c_*}(\cdot-c_*t)$, that acts as
\begin{align}\mathcal{L}_{c_*t,c_*}\eta=\begin{bmatrix}
        \eta_r \\
        \eta_p 
    \end{bmatrix}+ 2\begin{bmatrix}
r_{c_*}(\cdot-c_*t)\eta_r\\
0\end{bmatrix}.\label{eqn:localized}\end{align}
These time-dependent operators generate an exponentially stable evolution in weighted spaces (\Cref{prop:linearstability}), provided the neutral modes are avoided via the orthogonality conditions~\eqref{eqn:orthostability}. The term $ T^{1,0}$ introduces forcing to~\eqref{eqn:etalin} which preserves orthogonality with respect to $\phi_{c_*}(\cdot-c_*t)$. In particular,
\begin{align}\Omega\Big(\partial_\xi\phi_{c_*}(\cdot-c_* t),\eta_{1}(t)\Big)=\Omega\Big(\partial_c\phi_{c_*}(\cdot-c_*t),\eta_{1}(t)\Big)=0,\quad t\geq0. \label{eqn:etalinortho}\end{align}

Additionally, we consider a cruder `homogeneous' linear approximation of the tail $\eta(t)$ that neglects the localized linear term:
\begin{align}\label{eqn:etah}\dot{\eta}_1^{\mathrm{h}}(t)=\mathcal{J}\eta_1^{\mathrm{h}}(t)+\sigma T^{1,0}(c_* t,c_*)[\kappa].\end{align}
In comparison to~\eqref{eqn:etalin}, the linear operator $\mathcal{L}_{c_*t,c_*}$ has been replaced by $I$. This simplifies the analysis:~\eqref{eqn:etah} is a forced discrete wave equation, for which we can derive explicit solution formulas (\Cref{sec:radiative}). Despite the seemingly large difference between the definitions~\eqref{eqn:etalin} and~\eqref{eqn:etah}, the resulting `tails' do not differ much from each other. See \Cref{fig:relative}, which shows the difference $\|\eta_{1}(t)-\eta_1^{\mathrm{h}}(t)\|$ in proportion to $\|\eta_{1}(t)\|+\|\eta_1^{\mathrm{h}}(t)\|$ over time. Although the wave profile in~\eqref{eqn:localized} introduces secular modes, the map $T_\kappa$ avoids these by design. This structural aspect is shared by~\eqref{eqn:etah}, since the discrete wave equation $\dot{w}=\mathcal{J}w$ contains no secular modes.
\begin{figure}[H]
        \centering
        \includegraphics[width=0.6\linewidth]{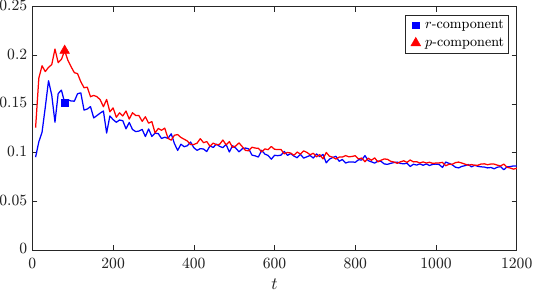}
        \caption{Relative difference $\frac{\|\eta_{1}(t)-\eta_1^{\mathrm{h}}(t)\|}{\|\eta_{1}(t)\|+\|\eta_1^{\mathrm{h}}(t)\|}$, for a particular realization with $c_*=1.015$ and $\kappa(i)\in \text{Unif}([-\sqrt{3},\sqrt{3}])$. The differences remain below 0.25, indicating a strong correspondence.}
        \label{fig:relative}
\end{figure}

\paragraph{Second-order Phase and Amplitude Corrections}
It has become clear that leading-order effects do not capture any attenuation of the amplitude $c(t)$. Next, we consider second-order contributions in the $\sigma$-expansion of the modulation parameters.  We identify $\gamma_2$ and $c_2$ in~\eqref{eqn:expgamma} and~\eqref{eqn:expc} by isolating $\mathcal{O}(\sigma^2)$ contributions in~\eqref{eqn:postl} and~\eqref{eqn:postc}. In particular, by expanding
\begin{align*}
    \begin{bmatrix}
         \Gamma^{1,0}(\xi(t),c(t))\\
         C^{1,0}(\xi(t),c(t))
     \end{bmatrix}
     =&\begin{bmatrix}
         \Gamma^{1,0}(c_*t,c_*)\\
         C^{1,0}(c_*t,c_*)
     \end{bmatrix}+\sigma\xi_1(t)\partial_\xi\begin{bmatrix}
         \Gamma^{1,0}(c_*t,c_*)\\
         C^{1,0}(c_*t,c_*)
     \end{bmatrix}+\sigma c_1(t)\partial_c\begin{bmatrix}
         \Gamma^{1,0}(c_*t,c_*)\\
         C^{1,0}(c_*t,c_*)
     \end{bmatrix}+\mathcal{O}(\sigma^2),
\end{align*}
we find that $\dot{\gamma}_2$ and $\dot{c}_2$ satisfy an equation of the form
\begin{align}
    \begin{bmatrix}
        \dot{\overline{\gamma}}_2(t)\\
        \dot{\overline{c}}_2(t)
    \end{bmatrix}=&(\gamma_1(t)+\int_0^t c_1(s) \ \d s)\partial_\xi\begin{bmatrix}
        \Gamma^{1,0}(c_*t,c_*)\\
        C^{1,0}(c_*t,c_*)
    \end{bmatrix}[\kappa]+ c_1(t)\partial_c\begin{bmatrix}
        \Gamma^{1,0}(c_*t,c_*)\\
        C^{1,0}(c_*t,c_*)
    \end{bmatrix}[\kappa]\nonumber\\
    &+\begin{bmatrix}
        \Gamma^{1,1}(c_*t,c_*)\\
        C^{1,1}(c_*t,c_*)
    \end{bmatrix}[\kappa,\overline{\eta}(t)]
    +\begin{bmatrix}
        \Gamma^{0,2}(c_*t,c_*)\\
        C^{0,2}(c_*t,c_*)
    \end{bmatrix}[\overline{\eta}_r(t),\overline{\eta}_r(t)],\label{eqn:overline}
\end{align}
with initial conditions $\overline{\gamma}_2(0)=0$ and $\overline{c}_2(0)=0$. The system~\eqref{eqn:overline} defines second-order corrections $(\overline{\gamma}_2,\overline{c}_2)$ based on a general linear approximation $\sigma \overline{\eta}$ to $\eta$. In particular, we retrieve the true second-order corrections $(\gamma_2,c_2)$ in~\eqref{eqn:expgamma} and~\eqref{eqn:expc} upon picking $\overline{\eta}=\eta_1$. Additionally, we obtain corrections $(\gamma_2^{\mathrm{h}},c_2^{\mathrm{h}})$ based on the homogeneous tail $\overline{\eta}=\eta_1^{\mathrm{h}}$ introduced in~\eqref{eqn:etah}.

As seen in \Cref{fig:second-order}, the second-order approximation $(\gamma_{2},c_{2})$ effectively captures the amplitude attenuation. The second-order corrections capture to a large extent how much energy radiates away from the soliton as a consequence of its shaking, leading to an $\mathcal{O}(\sigma^2)$ amplitude attenuation.
\begin{figure}[H]
    \centering
    \begin{subfigure}[t]{0.48\textwidth}
        \centering
        \includegraphics[width=\linewidth]{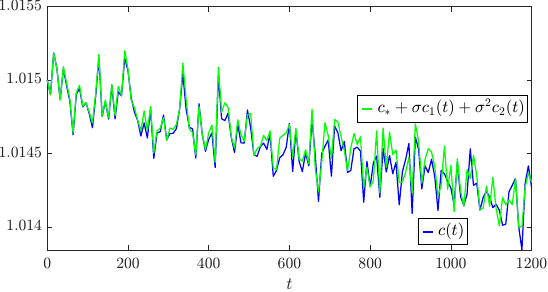}
    \end{subfigure}
    \hfill
    \begin{subfigure}[t]{0.48\textwidth}
        \centering
        \includegraphics[width=\linewidth]{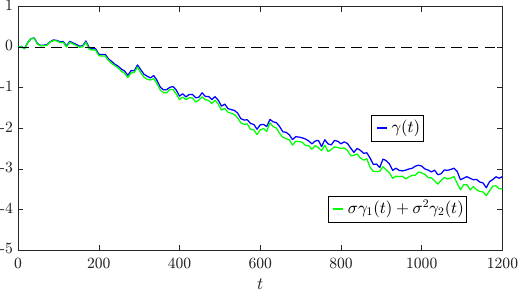}
    \end{subfigure}

    \caption{The second-order approximations $c_*+\sigma c_1(t)+\sigma^2{c}_{2}(t)$ (left) and $\sigma\gamma_1(t)+\sigma^2{\gamma}_{2}(t)$ (right) compared to $c(t)$ (left) and $\gamma(t)$ (right), respectively. ($\sigma=0.07$, $c_*=1.015$ and $\kappa(i)\in \text{Unif}([-\sqrt{3},\sqrt{3}])$.)}
    \label{fig:second-order}
\end{figure}

Although the tail approximation $\eta_1^\mathrm{h}$ defined through~\eqref{eqn:etah} neglects a linear term localized at the soliton location, 
 the second-order corrections $({\gamma}^{\mathrm{h}}_2,{c}^{\mathrm{h}}_2)$ based on $\eta_1^\mathrm{h}$ are 
 remarkably close 
 to the true corrections $({\gamma}_2,{c}_2)$---see \Cref{fig:2vsap}. We have thus arrived at two tractable approximations to the parameters $(\gamma,c,\eta)$ that capture attenuation. In the next sections, we analyze the linear tails $\overline{\eta}$ and resulting second-order corrections $(\overline{\gamma}_2,\overline{c}_2)$ in more detail.
\begin{figure}[H]
    \centering
    \begin{subfigure}[t]{0.48\textwidth}
        \centering
        \includegraphics[width=\linewidth]{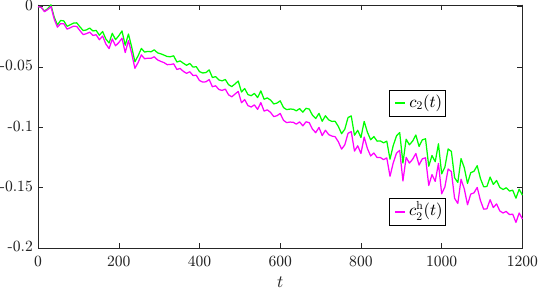}
    \end{subfigure}
    \hfill
    \begin{subfigure}[t]{0.48\textwidth}
        \centering
        \includegraphics[width=\linewidth]{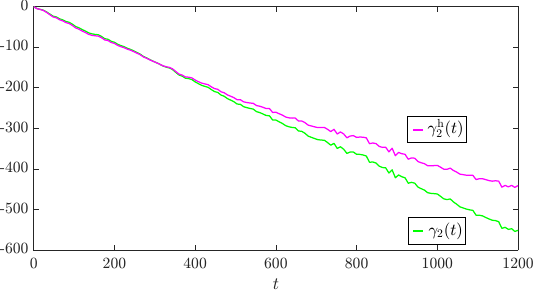}
    \end{subfigure}

    \caption{The second-order amplitude contribution $c_2(t)$ compared to $c^{\mathrm{h}}_2(t)$ (left) and the second-order phase contribution ${\gamma}_{2}(t)$ compared to ${\gamma}^{\mathrm{h}}_{2}(t)$ (right). ($c_*=1.015$ and $\kappa(i)\in \text{Unif}([-\sqrt{3},\sqrt{3}])$.)}
    \label{fig:2vsap}
\end{figure}

\section{Radiative Tail}\label{sec:radiative}
In this section, we examine the radiative tail that forms behind the solitary wave in more detail. In \Cref{sec:reduced}, we have introduced two approximations to the tail $\eta$ of the form
\begin{align}\dot{\overline{\eta}}(j,t)=\mathcal{J}\mathcal{A}_{c_*t,c_*}\ \overline{\eta}(j,t)+ T^{1,0}(c_* t,c_*)[\kappa](j),\quad \text{with}\quad\overline{\eta}(j,0)=0.\label{eqn:etaline}\end{align}
The approximation $\eta_{1}$ defined in~\eqref{eqn:etalin} features the soliton linearization $\mathcal{A}_{c_*t,c_*}=\mathcal{L}_{c_*t,c_*}$, while $\eta_1^{\mathrm{h}}$ is constructed using the constant-coefficient reduction $\mathcal{A}_{c_*t,c_*}=I$; see~\eqref{eqn:etah}.
Our aim here is to characterize the \textit{asymptotic} behavior of such tails $\overline{\eta}(t)$. In particular, we find that $\overline{\eta}(t)$ converges when viewed in a frame that moves along with the solitary wave.  Besides being interesting on its own, this can be linked directly  
to the long-time behavior of $\overline{c}_2(t)$. Indeed, this result allows us to derive a limit for the expected attenuation $\mathbb{E}[\overline{c}_2(t)]$ in \Cref{sec:attenuation}. We rely on the following assumptions on $\mathcal{A}_{c_*t,c_*}$, which are met by both $I$ and $\mathcal{L}_{c_*t,c_*}$:
\begin{hyp}\label{hyp:At}
We have $c_*\in(c_-,c_+)$ and 
$\{\mathcal{A}_{c_*t,c_*}\}_{t\in\R}$ is a family of bounded linear operators on $\ell^2(\Z;\R^2)$ that satisfies the spatio-temporal shift 
invariance
\begin{align}[\mathcal{A}_{c_*t,c_*}\eta](j+d)=[\mathcal{A}_{c_*t-d,c_*}\eta(\cdot+d)](j),\quad j,d\in\Z,\ t\in\R.\label{eqn:invariance}\end{align}
The operators $\{\mathcal{J}\mathcal{A}_{c_*t,c_*}\}_{t\in\R}$ generate an evolution family $\{U_{c_*}(t,s)\}_{t\geq s }$ on $\ell^2(\Z;\R^2)$. There exist constants $K,b>0$  such that
    \begin{align}\|e^{a(\cdot-c_*t)}U_{c_*}(t,s)w\|_{\ell^2(\Z;\R^2)}\leq Ke^{-b(t-s)}\|e^{a(\cdot-c_*s)}w\|_{\ell^2(\Z;\R^2)}, \quad t\geq s,\label{eqn:exp}\end{align}
    provided
\begin{align}\Omega(\partial_\xi\phi_{c_*}(\cdot-c_*s),w)=\Omega(\partial_c\phi_{c_*}(\cdot-c_*s),w)=0.\end{align}
\end{hyp}
We observe numerically (\Cref{subfig:t=40}) that the tail $\eta(t)$ starts out as zero, and, roughly speaking, develops in the interval $[-c_*t,c_*t]$ as the soliton passes through and excites lattice sites. In particular, when viewed in a co-moving frame traveling rightward with velocity $c_*$, the tail is roughly active in the interval $[-2c_*t,0]$. This leads to a limiting situation where all lattice sites behind the soliton are excited, while all lattice sites ahead remain almost unperturbed. Our tail approximations $\overline{\eta}$ allow us to characterize this limiting situation in more detail. In terms of the evolution family,~\eqref{eqn:etaline} reads
\[\overline{\eta}(t)=\int_0^t U_{c_*}(t,s) T^{1,0}(c_* s,c_*)[\kappa]\ \d s,\]
in which the forcing term $ T^{1,0}(c_* s,c_*)$ is localized around the soliton site $j\approx c_*s$ as it propagates from $0$ to $c_*t$. At times
\[t_n=(n+p)/c_*, \quad n\in\N,\ p\in [0,1),\]
we shift the tail by an integer $\lfloor c_*t_n\rfloor=n$, which leads to the identity 
\begin{align*}\overline{\eta}(\cdot+n,t_n)=&\int_{0}^{t_n} U_{c_*}(p/c_*,p/c_*-\tau)T^{1,0}(p-c_*\tau,c_*)[\kappa(\cdot+n)]\ \d \tau.
\end{align*}
We derive such representations using the Green's function  associated to $U_{c_*}$, which we analyze in \Cref{subsec:linear} ahead. The only dependence of the integrand on $n$ is through the shift of the random coefficients $\kappa(\cdot +n)$,
which does not affect their distribution.
Combined with the exponential stability of the evolution family $U_{c_*}$, this allows us to take the limit $n \to \infty$ and show that this shifted tail approaches an equilibrium distribution.
\begin{proposition}\label{prop:stationary}
Assuming \Cref{hyp:coeff} and \Cref{hyp:At}, there exists a (random) function $\overline{\eta}^\infty:\R\to\R$ such that for each $j\in\Z$ and $p\in[0,1)$, the sequence of random variables \[\Big(\overline{\eta}(j+n,(n+p)/c_*)\Big)_{n\geq 1}\] 
converges in distribution to $\overline{\eta}^\infty(j-p)$. 
\end{proposition}
\begin{remark}\label{rem:continuity}
A representation of $\overline{\eta}(x)$ is given in~\eqref{eqn:etainf} ahead, which we suspect is (almost surely) continuous in $x$; see \Cref{fig:tail} ahead. We do not attempt to prove this here.
\end{remark}
The proof of \Cref{prop:stationary} is given at the end of \Cref{subsec:asymptotics}. In the limit, the tail is still $1/c_*$-periodic in time in some sense, reflecting that the lattice periodicity is traversed with velocity $c_*$. Here and below, $p\in[0,1)$ denotes the phase of the soliton (for instance, $p=0$ means centered at a lattice point and $p=1/2$ means exactly between lattice points). The convergence in \Cref{prop:stationary} holds only in distribution because, roughly speaking, the solitary wave continually encounters different random coefficients. Nevertheless, the distribution of the random coefficients is translation-invariant.

\subsection{Linear theory}\label{subsec:linear}
Our strategy will be to work with a decomposition of $\overline{\eta}$ in terms of deterministic functions. More precisely, we analyze the asymptotic behavior of tails $\overline{\eta}(t)$  based on a `response' function $R:\Z\times \Z \times \R^+\to\R^2$, defined through
\begin{align}\dot{\overline{R}}(j,m,t)=\mathcal{J}\mathcal{A}_{c_*t,c_*}\overline{R}(j,m,t)+ T^{1,0}(c_* t,c_*)[\delta_m](j),\quad \text{with}\quad \overline{R}(j,m,0)=0.\label{eqn:response}\end{align} 
Here, $\delta_m\in\ell^2(\Z;\R)$ denotes the sequence with value $1$ at $j=m$ and 0 elsewhere. Hence, the response function expresses the effect of a spring heterogeneity at site $m$ on the random tail $\overline{\eta}(t)$. The advantage of this decomposition is that we can analyze the asymptotics of the deterministic components $\overline{R}$ in a co-moving frame, without directly involving the constantly shifting random coefficients $\kappa$. We recall that $T^{1,0}$ is introduced in~\eqref{eqn:Tk}, and remark that
\begin{align}
T^{1,0}(c_* t,c_*)[\delta_m](j)=& \frac{1}{2}\begin{bmatrix}
      \partial_\xi\phi_{c_*}(j-c_*t) \\
      \partial_c\phi_{c_*}(j-c_*t)
    \end{bmatrix}^\top  A^{-1}(c_*)\begin{bmatrix}
         \partial_\xi r^2_{c_*}(m-c_*t)\\
         \partial_cr^2_{c_*}(m-c_*t)
    \end{bmatrix}\nonumber
\\&+ r_{c_*}(m-c_*t)\begin{bmatrix}0\\
     (\delta_m-\delta_{m+1})(j) 
\end{bmatrix}.\label{eqn:Tdelta}
\end{align}
Since the wave-profiles and their derivatives are exponentially localized \cite[Proposition 5.5]{FPI}, this forcing term satisfies the bound
\begin{align}\Big|T^{1,0}(c_*t,c_*)[\delta_m](j)\Big|\leq  Ce^{-\beta|j-c_*t|} 
     e^{-\beta|m-c_*t|}+Ce^{-\beta|m-c_*t|}
     (\delta_0+\delta_{1})(j-m)\label{eqn:asymptotics}\end{align} 
for some $C>0$ and $\beta>a$, where $a$ is the weight of \Cref{prop:linearstability}. 
Below, we derive the representation of $\overline{\eta}(t)$ as a random series. To this end, we note that the solution to~\eqref{eqn:response} can be represented by the Duhamel formula
\begin{align}\overline{R}(\cdot,m,t)=\int_0^t U_{c_*}(t,s) T^{1,0}(c_*s,c_*)[\delta_m]\ \d s. \label{eqn:duhamel}\end{align}
\begin{proposition}\label{prop:randomsum}
 Assume \Cref{hyp:coeff} and \Cref{hyp:At}. For each $t\geq0$, the random sequence
\begin{align}
    \overline{\eta}(j,t)=&\sum_{m\in\Z}\kappa(m)\overline{R}(j,m,t),\quad j\in\Z,\label{eqn:sum}
\end{align}lies almost surely in $\ell^2(\Z;\R^2)$ and solves the linear system~\eqref{eqn:etaline}.\end{proposition}
\begin{proof}
We first show that the series in~\eqref{eqn:sum} lies in $\ell^2(\Z;\R^2)$. For each $m\in\Z$:
\begin{align*}
    \|\overline{R}(\cdot,m,t)\|_{\ell^2(\Z;\R^2)}=&\Big\|\int_0^t U_{c_*}(t,s) T^{1,0}(c_*s,c_*)[\delta_m]\ \d s\Big\|_{\ell^2(\Z;\R^2)}\\
    \leq & t \sup_{s\in[0,t]}\|U_{c_*}(t,s)\|_{\mathcal{L}(\ell^2(\Z;\R^2))}\|T^{1,0}(c_*s,c_*)[\delta_m]\|_{\ell^2(\Z;\R^2)}.
\end{align*}
From~\eqref{eqn:asymptotics} we then get
\begin{align*}
    \|T^1(c_*s,c_*)[\delta_m]\|_{\ell^2(\Z;\R^2)}\leq  {C}e^{-\beta|m-c_*s|}\sum_{j\in\Z} e^{-\beta|j-c_*s|} 
     +{C}e^{-\beta|m-c_*s|}\leq {\tilde{C}}e^{-\beta|m-c_*s|},
\end{align*}
which is summable in $m$:
\[\Big\|\sum_{m\in\Z}\kappa(m)\overline{R}(\cdot,m,t)\Big\|_{\ell^2(\Z;\R^2)}\leq\alpha\sum_{m\in\Z}\|\overline{R}(\cdot,m,t)\|_{\ell^2(\Z;\R^2)}<\infty.\]
The variation of constants formula then shows that~\eqref{eqn:etaline} is solved by 
\begin{align*}
    \overline{\eta}(j,t)=&\int_0^t U_{c_*}(t,s)T^{1,0}(c_*s,c_*)[\kappa]\ \d s.\end{align*}
    The result follows via~\eqref{eqn:duhamel} and the identity
    \[T^{1,0}(c_*s,c_*)[\kappa]=\sum_{m\in\Z}\kappa(m)T^{1,0}(c_*s,c_*)[\delta_m].\qedhere\]
    \end{proof}
We proceed with a result regarding the Green's function associated to $\mathcal{J}\mathcal{A}_{c_*t,c_*}$, which will be key to characterizing the asymptotics of the response function $\overline{R}(j,m,t)$. 
\begin{lemma}[See \Cref{app:greens}]\label{lem:greens}
   Assuming \Cref{hyp:At}, there exists a Green's function $\mathcal{G}_{c_*}:\R^+\times\R\times\R\to \R^{2\times2}$ such that the evolution family associated to $\mathcal{J}\mathcal{A}_{c_*t,c_*}$ admits the representation 
    \begin{align}U_{c_*}(t,s)\begin{bmatrix}
        u \\
        v
    \end{bmatrix}(j)=\sum_{k\in\Z}\mathcal{G}_{c_*}(t-s,j-c_*t,k-c_*s)\begin{bmatrix}
        u(k)\\
        v(k)
    \end{bmatrix},\quad \begin{bmatrix}
        u \\ v
    \end{bmatrix}\in\ell^2(\Z;\R^2).\label{eqn:form}\end{align}
\end{lemma}

\begin{remark}[See \Cref{lem:discwave}]
In the constant-coefficient case $\mathcal{J}\mathcal{A}_{c_*t,c_*}=\mathcal{J}$, the evolution family $U_{c_*}$ reduces to the unitary $C_0$-group $\{e^{\mathcal{J}t}\}_{t\in \R}$. In particular, these solution operators admit an explicit kernel representation 
\begin{align}
   U_{c_*}(t,s)\begin{bmatrix}
        u\\
        v
    \end{bmatrix}(j)=e^{\mathcal{J}(t-s)} \begin{bmatrix}
        u\\
        v
    \end{bmatrix}(j)=\sum_{k\in\Z}\Phi(j-k,t-s)\begin{bmatrix}
        u(k)
\\
v(k)\end{bmatrix},\quad j\in\Z,\ t\in \R,\label{eqn:linearwave}
\end{align}
with
\begin{align}\label{eqn:kernel}
    \Phi(j,t):=\begin{bmatrix}
        J_{2j}(2t) & -J_{2j+1}(2t)\\
        -J_{2j-1}(2t) & J_{2j}(2t)
    \end{bmatrix}.
\end{align}
Here, $J_n:\R\to \R$ are Bessel functions of the first kind, and for $n\in \Z^-$ we use the convention that $J_n(x)=(-1)^nJ_{-n}(x)$. Hence, the Green's function associated to the discrete wave operator $\mathcal{J}$ satisfies~\eqref{eqn:form}, with
\[\mathcal{G}_{c_*}(\alpha,\beta,\gamma)=\Phi(\beta-\gamma+c_*\alpha,\alpha).\]
\end{remark}

\subsection{Asymptotic response}\label{subsec:asymptotics}

We proceed by analyzing the asymptotic behavior of~\eqref{eqn:response} in the co-moving frame, based on the Green's function representation~\eqref{eqn:form}. We will see that at times
\[t_n=(n+p)/c_*, \quad n\in\N,\ p\in [0,1),\]
we have a representation of the (shifted) response function $\overline{R}(j+\lfloor c_*t\rfloor,m+\lfloor c_* t\rfloor,t)$ as
\begin{align*}\overline{R}(j+n,m+n,t_n)=&\int_{0}^{t_n} \Big(U_{c_*}(p/c_*,p/c_*-\tau)T^{1,0}(p-c_*\tau,c_*)[\delta_m]\Big)(j)\ \d \tau\\
=&\int_{0}^{t_n} \sum_{k\in\Z}\mathcal{G}_{c_*}(\tau,j-p,k-p+c_*\tau)T^{1,0}(p-c_*\tau,c_*)[\delta_m](k)\ \d \tau.
\end{align*}
Hence, this collects the effect of a heterogeneity at site $m$ on site $j$ as the solitary wave travels from $-n$ to $p$.  In this `co-moving frame', stationary forcing competes with an exponentially stable linear flow.
This allows us to take $n \to \infty$ and arrive
at the limiting response function
$\overline{R}^\infty:\Z\times\Z\times[0,1]\to\R^2$, defined as
\begin{align}\overline{R}^\infty(j,m,p)=\int_{0}^\infty \sum_{k\in\Z}\mathcal{G}_{c_*}(\tau,j-p,k-p+c_*\tau)T^{1,0}(p-c_*\tau,c_*)[\delta_m](k)\ \d \tau.\label{eqn:Rinf}\end{align}
Note that in the constant-coefficient case (\Cref{lem:discwave}), the kernel in~\eqref{eqn:Rinf} is explicitly given by
\[\mathcal{G}_{c_*}(\tau,j-p,k-p+c_*\tau)=\Phi(j-k,\tau).\]
\begin{lemma}
Assume \Cref{hyp:At}. For each $j,m\in\Z$ and $p\in[0,1)$,
we have the pointwise limit
    \begin{align}\lim_{n\to\infty}\overline{R}(j+n,m+n,(n+p)/c_*) =\overline{R}^\infty(j,m,p).\label{eqn:pointwise}\end{align}
\end{lemma}
\begin{proof}
Let $p\in [0,1)$ and \[t_n=(n+p)/c_*, \quad n\in\N.\] Then~\eqref{eqn:duhamel} gives
\begin{align*}\overline{R}(j+n,m+n,t_n)=&\int_0^{t_n} \sum_{k\in\Z}\mathcal{G}_{c_*}(t_n-s,j+n-c_*t_n,k-c_*s)T^{1,0}_{\delta_{m+n}}(c_*s,c_*)(k)\ \d s\\
=&\int_0^{t_n} \sum_{k'\in\Z}\mathcal{G}_{c_*}(t_n-s,j-p,k'+n-c_*s)T^{1,0}(c_*s,c_*)[\delta_{m+n}](k'+n)\ \d s.
\end{align*}
Using
\begin{align*}T^{1,0}(c_*s,c_*)[\delta_{m+n}](k'+n)=&T^{1,0}(c_*s-n,c_*)[\delta_m](k')
=T^{1,0}(p-c_*(t_n-s),c_*)[\delta_m](k'),\
\end{align*}
and substituting $\tau=t_n-s$, we obtain that
\begin{align}\overline{R}(j+n,m+n,t_n)=&\int_0^{t_n} \sum_{k'\in\Z}\mathcal{G}_{c_*}(\tau,j-p,k'-p+c_*\tau)T^{1,0}(p-c_*\tau,c_*)[\delta_m](k')\ \d \tau\label{eqn:shiftgreen}\\
\to& \overline{R}^\infty(j,m,p) \quad \text{as}\quad n\to \infty.\nonumber\end{align}
Thus, we have established~\eqref{eqn:pointwise}: pointwise convergence in $j$ and $m$. 
\end{proof}
 We will in fact require a stronger form of convergence for our purposes of analyzing the asymptotics of $\mathbb{E}[\overline{c}_2(t)]$ in \Cref{sec:attenuation}. Recall that the weight $a>0$ is introduced in \Cref{prop:linearstability}. We then define the exponentially weighted spaces
\[
\ell^2_a(\mathbb{Z}^2;\mathbb{R}^2)
:= \left\{ f:\mathbb{Z}^2\to\mathbb{R}^2 \;\middle|\; \|f\|_{\ell^2_a(\mathbb{Z}^2;\mathbb{R}^2)}<\infty \right\},\]
with norm
\[
\|f\|_{\ell^2_a(\mathbb{Z}^2;\mathbb{R}^2)}^2
= \sum_{m\in\mathbb{Z}} \sum_{j\in\Z}e^{2aj} \big(f_r^2(j,m)+f^2_p(j,m)\big) .
\]
\begin{proposition}
\label{prop:Rinf}
Assume \Cref{hyp:At}. The response function $\overline{R}(\cdot,\cdot,(n+p)/c_*)$ converges to $\overline{R}^\infty(\cdot,\cdot,p)$ in $\ell^2_a(\mathbb{Z}^2;\mathbb{R}^2)$ at an exponential rate: there exist constants $C,q>0$ such that
    \begin{align}\|\overline{R}(\cdot+n,\cdot+n,(n+p)/c_*)- \overline{R}^\infty(\cdot,\cdot,p)\|_{\ell^2_a(\mathbb{Z}^2;\mathbb{R}^2)}\leq Ce^{-qn},\quad n\in\N.\label{eqn:exprate}\end{align}
\end{proposition}
Before giving the proof, we prepare two intermediate results. The first regards the weighted norm of the forcing term in~\eqref{eqn:Rinf}.
\begin{lemma}\label{lem:Tbound}Assuming \Cref{hyp:At}, there exist constants $C,\gamma>0$, such that for each $p\in[0,1)$, $m\in\Z$ and $\tau\geq0$, we have
    \begin{align*}
    \|e^{a(\cdot+c_*\tau) }T^{1,0}(p-c_*\tau,c_*)[\delta_m]\|_{\ell^2(\Z;\R^2)}\leq \tilde{C}e^{-\gamma|m+c_*\tau|}.
\end{align*}
\end{lemma}
\begin{proof}
    The estimate~\eqref{eqn:asymptotics} provides
\begin{align*}
    \|e^{a(\cdot+c_*\tau) }T^{1,0}(p-c_*\tau,c_*)[\delta_m]\|_{\ell^2(\Z;\R^2)}\leq& Ce^{-\beta|m+c_*\tau|}\|e^{a(\cdot+c_*\tau) }e^{-\beta|\cdot+c_*\tau|}\|_{\ell^2(\Z;\R)} 
     \\
     &+Ce^{-\beta|m+c_*\tau|}
     \|e^{a(\cdot+c_*\tau) }(\delta_0+\delta_{1})(\cdot-m)\|_{\ell^2(\Z;\R)},\end{align*}
where we note that
\[e^{a(\cdot+c_*\tau) }e^{-\beta|\cdot+c_*\tau|}\in\ell^2(\Z;\R),\]
since $\gamma:=\beta-a>0$. We furthermore compute
\[\|e^{a(\cdot+c_*\tau) }(\delta_0+\delta_{1})(\cdot-m)\|_{\ell^2(\Z;\R)}=\Big(\sum_{j\in\Z}e^{2a(\cdot+c_*\tau) }(\delta_0+\delta_{1})(j-m)\Big)^{1/2}=(1+e^{2a})^{1/2}e^{a(m+c_*\tau)}, \]
and the result follows.
\end{proof}
We proceed by asserting that both $\overline{R}(\cdot,\cdot,t)$ and $\overline{R}^\infty(\cdot,\cdot,p)$ are well-defined in $\ell^2_a(\Z^2;\R^2)$.
\begin{lemma}\label{lem:contained}
Assume \Cref{hyp:At}. For each $t\geq 0$ and $p\in[0,1)$, the response function $\overline{R}(\cdot,\cdot,t)$ and its limit $\overline{R}^\infty(\cdot,\cdot,p)$ are contained in  $\ell^2_a(\mathbb{Z}^2;\mathbb{R}^2)$.
\end{lemma}
\begin{proof}
 From~\eqref{eqn:Rinf}, we compute
\begin{align*}
    \| \overline{R}^\infty(\cdot,m,p)\|^2_{\ell^2_a(\Z^2;\R^2)}
    =&\sum_{m\in \Z}\Big\|\int_{0}^{\infty} \sum_{k\in\Z}e^{a\cdot}\mathcal{G}_{c_*}(\tau,j-p,k-p+c_*\tau)T^{1,0}(p-c_*\tau,c_*)[\delta_m](k)\ \d \tau\Big\|^2_{\ell^2(\Z;\R^2)}\\
    =&\sum_{m\in \Z}\Big\|\int_{0}^{\infty} e^{a\cdot}U_{c_*}(p/c_*,p/c_*-\tau)T^{1,0}(p-c_*\tau,c_*)[\delta_m]\ \d \tau\Big\|_{\ell^2(\Z;\R^2)}^2.
    \end{align*}
    We then apply the exponential stability bound~\eqref{eqn:exp}: 
    \begin{align*}
    \| \overline{R}^\infty(\cdot,m,p)\|^2_{\ell^2_a(\Z^2;\R^2)}
    \leq & \sum_{m\in \Z}\Big(\int_{0}^{\infty} \|e^{a\cdot}U_{c_*}(p/c_*,p/c_*-\tau)T^{1,0}(-c_*\tau,c_*)[\delta_m]\|_{\ell^2(\Z;\R^2)}\ \d \tau\Big)^2\\
    \leq &K^2\sum_{m\in \Z}\Big(\int_{0}^{\infty}e^{-b\tau} \|e^{a(\cdot+c_*\tau)}T^{1,0}(p-c_*\tau,c_*)[\delta_m]\|_{\ell^2(\Z;\R^2)}\ \d \tau\Big)^2.
\end{align*}
Via \Cref{lem:Tbound}, we then obtain
\begin{align*}
 \| \overline{R}^\infty(\cdot,m,p)\|^2_{\ell^2_a(\Z^2;\R^2)}
  \leq  K^2\tilde{C}^2 \sum_{m\in\Z}\Big(\int_{0}^\infty e^{-b\tau}e^{-\gamma|m+c_*\tau|}\d \tau\Big)^2.
\end{align*}
A straightforward computation shows that for each $t\geq 0$
\begin{align}
\int_{t}^{\infty} e^{-b\tau}\,e^{-\gamma|m+c_*\tau|}\d\tau =
\begin{cases}
\dfrac{e^{-\gamma(c_*t+m)}\,e^{-bt}}{b/c_*+\gamma}, & c_*t+m \ge 0, \\[6pt]
\dfrac{e^{bm/c_*}-e^{\gamma(c_*t+m)}e^{-bt}}{\gamma-b/c_*}+\dfrac{e^{bm/c_*}}{b/c_*+\gamma}, & c_*t+m \le 0.
\end{cases}\label{eqn:cases}
\end{align}
In particular,
\begin{align*}
    \sum_{m\in\Z}\Big(\int_{0}^\infty e^{-b\tau}e^{-\gamma|m+c_*\tau|}\d \tau\Big)^2 \leq \tilde{C}\Big(\sum_{m=0}^\infty e^{-2m}+e^{-2bm}\Big)<\infty.
\end{align*}
We have thus shown that $\overline{R}^\infty(\cdot,\cdot,p)$ lies in $\ell^2_a(\Z^2;\R^2)$. An analogous computation based on~\eqref{eqn:shiftgreen} shows that $\overline{R}(\cdot+n,\cdot+n,t_n)$---and consequently its translate $\overline{R}(\cdot,\cdot,t_n)$---is also contained in $\ell^2_a(\Z^2;\R^2)$.
\end{proof}
We are then ready to prove \Cref{prop:Rinf}.
\begin{proof}[Proof of \Cref{prop:Rinf}]
We apply the same estimates as in \Cref{lem:contained} to the difference
\begin{align*}
    \mathcal{R}_n:=&\sum_{m\in \Z}\|e^{a\cdot}\big(\overline{R}(\cdot+n,m+n,t_n)- \overline{R}^\infty(\cdot,m,p)\big)\|^2_{\ell^2(\Z;\R^2)}\\
    =&\sum_{m\in \Z}\Big\|\int_{t_n}^{\infty} e^{a\cdot}U_{c_*}(p/c_*,p/c_*-\tau)T^{1,0}(p-c_*\tau,c_*)[\delta_m]\ \d \tau\Big\|_{\ell^2(\Z;\R^2)}^2.
    \end{align*}
Applying the stability bound~\eqref{eqn:exp} gives
\begin{align*}
    \mathcal{R}_n\leq &K^2\sum_{m\in \Z}\Big(\int_{t_n}^{\infty}e^{-b\tau} \|e^{a(\cdot+c_*\tau)}T^{1,0}(p-c_*\tau,c_*)[\delta_m]\|_{\ell^2(\Z;\R^2)}\ \d \tau\Big)^2\\
\leq & K^2\tilde{C}^2 \sum_{m\in\Z}\Big(\int_{t_n}^\infty e^{-b\tau}e^{-\gamma|m+c_*\tau|}\d \tau\Big)^2.
\end{align*}
Inspecting~\eqref{eqn:cases}, we now arrive at
\begin{align*}
    \sum_{m\in\Z}\Big(\int_{t}^\infty e^{-b\tau}e^{-\gamma|m+c_*\tau|}\d \tau\Big)^2 \leq& \tilde{C}\Big(e^{-2bt}\sum_{m\geq-t} e^{-2(t+m)}+\sum_{m\leq-t}e^{2bm}+e^{-2bt}\sum_{m\leq-t}e^{2(t+m)}\Big)\\
    \leq &3\tilde{C}\frac{e^{2b}}{e^{2b}-1}e^{-2bt},
\end{align*}
which shows~\eqref{eqn:exprate}.
\end{proof}

\subsection{Asymptotic tail}
Following the convergence of $\overline{R}$ to $\overline{R}^\infty$, we now turn to  \Cref{prop:stationary}, concerning the convergence of the shifted tail to $\overline{\eta}^\infty$. We will see that $\overline{\eta}^\infty$ can be represented by the random series
\begin{align}\overline{\eta}^\infty(x)=\sum_{m\in\Z}\zeta(m- \lceil x\rceil)\overline{R}^\infty(\lceil x \rceil,m,\lceil x \rceil-x),\quad x\in\R,\label{eqn:etainf}\end{align}
where $\zeta(j)$, $j\in\Z$ is an i.i.d. sequence with the same distribution as the random coefficients $\kappa(j)$. Regarding the continuity of this representation (\Cref{rem:continuity}), we note that
\[\overline{R}^\infty(j,m,1)=\overline{R}^\infty(j-1,m-1,0),\quad j,m\in\Z,\]
which shows that the specific representation~\eqref{eqn:etainf} satisfies
\[\lim_{x\downarrow j}\overline{\eta}^\infty(x)=\overline{\eta}^\infty(j),\quad j\in\Z.\]
In between lattice points $x\notin \Z$, the continuity of $\overline{\eta}^\infty$ is determined by that of $\overline{R}^\infty(j,m,p)$ in the $p$-variable, for which we do not pursue a proof here. The covariance of $\overline{\eta}^\infty=(\overline{\eta}_r^\infty,\ \overline{\eta}_p^\infty)^\top$ is given by
\[\mathbb{E}\big(\overline{\eta}_{r}^\infty(x)\overline{\eta}_{r}^\infty(y)\big)=\Big\langle \overline{R}_{r}^\infty(\lfloor x\rfloor,\cdot,x-\lfloor x\rfloor),\overline{R}_{r}^\infty(\lfloor y\rfloor,\cdot,y-\lfloor y\rfloor)\Big\rangle_{\ell^2(\Z;\R)},\]
with an analogous expression for $\overline{\eta}_{p}^\infty$. See \Cref{fig:tail} for a numerical evaluation.
\begin{figure}[h]
    \centering
    \begin{subfigure}[t]{0.48\textwidth}
        \centering
        \includegraphics[width=\linewidth]{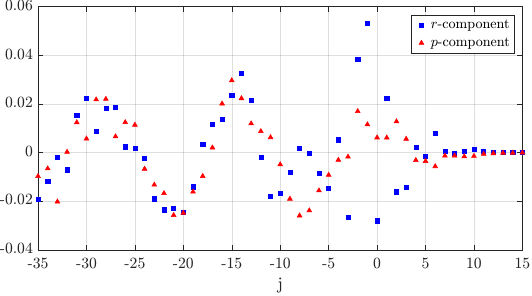}
    \end{subfigure}
    \hfill
    \begin{subfigure}[t]{0.48\textwidth}
        \centering
        \includegraphics[width=\linewidth]{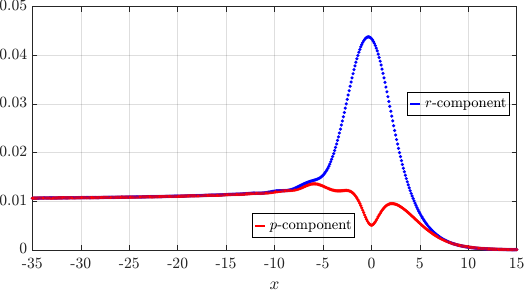}
    \end{subfigure}

    \caption{Realization (left) and standard deviation (right) of the limiting tail $\overline{\eta}^\infty$, based on $\overline{\eta}=\eta_1^\mathrm{h}$. ( $c_*=1.015$ and $\zeta(i)\in \text{Unif}([-\sqrt{3},\sqrt{3}])$.) }
    \label{fig:tail}
\end{figure}
We are then ready to prove the main result of this section.
\begin{proof}[Proof of \Cref{prop:stationary}]
Once again, we pick $p\in[0,1)$ and write $t_n=(n+p)/c_*$, so that at $x=j-p$, the representation~\eqref{eqn:etainf} reads
\[\overline{\eta}^\infty(j-p)
=\sum_{m\in\Z}\zeta(m- j)\overline{R}^\infty(j,m,p),\]
where $\zeta(j)$, $j\in\Z$ is an i.i.d. sequence with the same distribution as the random coefficients $\kappa(j)$.
We introduce
\[Y_n(j)= \sum_{m\in\Z}\zeta(m-j)\overline{R}(j+n,m+n,t_n),\]
which, for each $n\in\N$ and $j\in\Z$, has the same law as
\[\overline{\eta}(j+n,t_n)=\sum_{m\in\Z}\kappa(m+n)\overline{R}(j+n,m+n,t_n).\] 
Indeed, the i.i.d. sequences $\kappa$ and $\zeta$ follow the same distribution. 
We furthermore have
    \begin{align*}
    &Y_n(j)- \overline{\eta}^\infty(j-p)
  =\sum_{m\in\Z}\zeta(m-j)\Big(\overline{R}(j+n,m+n,t_n)-\overline{R}^\infty(j,m,p)\Big)
\end{align*}
and as a consequence of \Cref{prop:Rinf}:
\begin{align*}
    &\mathbb{E}[(Y_n(j)- \overline{\eta}^\infty(j-p))^2]
  =\sum_{m\in\Z}\Big(\overline{R}(j+n,m+n,t_n)-\overline{R}^\infty(j,m,p)\Big)^2\to 0
\end{align*}
as $n\to \infty$. In particular, $Y_n(j)\stackrel{d}{\longrightarrow} \overline{\eta}^\infty(j-p)$. It follows that also \[ \overline{\eta}(j+n,t_n)\stackrel{d}{\longrightarrow}\overline{\eta}^\infty(j-p).\qedhere\]
\end{proof}

\section{Amplitude Attenuation}\label{sec:attenuation}
We finally return to the main goal of this work: capturing the amplitude attenuation that affects the propagation of solitary waves in~\eqref{eqn:FPU}. We build on the results of \Cref{sec:radiative}, where we analyzed linear tail approximations $\overline{\eta}$ based on a general time-dependent linearization $\mathcal{J}\mathcal{A}_{c_*t,c_*}$; see~\eqref{eqn:etaline} and \Cref{hyp:At}. The representation of $\overline{\eta}(t)$ as a random series (\Cref{prop:randomsum}) allows us to compute expectations of the second-order corrections $\overline{\gamma}_2(t)$ and $\overline{c}_2(t)$, defined via~\eqref{eqn:overline}. Although the random coefficients $\kappa$---and by extension the linear tail $\overline{\eta}(t)$---are mean-zero, the quadratic terms in~\eqref{eqn:overline} lead to a non-zero expectation of $\dot{\overline{\gamma}}_2(t)$ and  $\dot{\overline{c}}_2(t)$.

In \Cref{prop:expectation} below, we give an explicit representation for the amplitude attenuation via the bilinear maps $\Gamma^{1,1}, \Gamma^{0,2}, C^{1,1}$ and $C^{0,2}$ introduced in~\eqref{eqn:bilinear1}--\eqref{eqn:bilinear2}, and the response function $\overline{R}$ introduced in~\eqref{eqn:response}. Recall that $\ell^2_a(\mathbb{Z}^2;\mathbb{R}^2)$ denotes the exponentially weighted spaces
\[
\ell^2_a(\mathbb{Z}^2;\mathbb{R}^2)
:= \left\{ f:\mathbb{Z}^2\to\mathbb{R}^2 \;\middle|\; \|f\|_{\ell^2_a(\mathbb{Z}^2;\mathbb{R}^2)}<\infty \right\},\]
with norm
\[
\|f\|_{\ell^2_a(\mathbb{Z}^2;\mathbb{R}^2)}^2
= \sum_{m\in\mathbb{Z}} \|e^{a\cdot} f(\cdot,m)\|_{\ell^2(\Z;\R^2)}^2 .
\]
For $\xi\in\R$, $c\in(c_-,c_+)$, we define the linear map $M^{1,1}(\xi,c):\ell^2_a(\Z^2;\R^2)\to\R^2$ via
\begin{align}M^{1,1}(\xi,c)[\alpha]=\sum_{m\in\Z}\begin{bmatrix}
        \Gamma^{1,1}(\xi,c)\\
        C^{1,1}(\xi,c)
    \end{bmatrix}[\delta_m,\alpha(\cdot,m)],\label{eqn:MLdef}
\end{align}
and the bilinear map $M^{0,2}(\xi,c):\ell^2_a(\Z^2;\R^2)\times \ell^2_a(\Z^2;\R^2)\to\R^2$ by
\begin{align}M^{0,2}(\xi,c)[\alpha,\beta]=\sum_{m\in\Z}\begin{bmatrix}
        \Gamma^{0,2}(\xi,c)\\
        C^{0,2}(\xi,c)
    \end{bmatrix}[\alpha_{r}(\cdot,m),\beta_{r}(\cdot,m)].\label{eqn:MQdef}
\end{align}
Both $M^{1,1}$ and $M^{0,2}$ are bounded (see \Cref{app:continuity}) and we recall that $\overline{R}(\cdot,\cdot,t)$ is contained in $\ell^2_a(\Z^2;\R^2)$ for each $t\geq0$ on account of \Cref{lem:contained}. Finally, we define $M^{1,0}_{\RomanI}(\xi,c),\ldots,M^{1,0}_{\RomanIV}(\xi,c)$ as the correlations
\begin{align}M^{1,0}_{\RomanI}(\xi,c)=&-  \frac{\alpha^{-2}_0(c)\alpha_1(c)}{2c}\mathbb{E}\ \sum_{j\in\Z}\kappa(j)[r^2_{c}(j-\xi)-r^2_{c}(j)]\partial_\xi\begin{bmatrix}
        \Gamma^{1,0}(\xi,c)\\
        C^{1,0}(\xi,c)
    \end{bmatrix}[\kappa],\label{eqn:corM1}\\
    M^{1,0}_{\RomanII}(\xi,c)=&\frac{1}{2c\alpha_0(c)}\mathbb{E}\ \sum_{j\in\Z}\kappa(j)\int_0^{\xi}\partial_c r^2_{c}(j-s)\d s \ \partial_\xi\begin{bmatrix}
        \Gamma^{1,0}(\xi,c)\\
        C^{1,0}(\xi,c)
    \end{bmatrix}[\kappa],\label{eqn:corM2}\\
    M^{1,0}_{\RomanIII}(\xi,c)=&- \frac{\alpha_0^{-1}(c)}{2c^2}\mathbb{E}\  \sum_{j\in\Z}\kappa(j)\Big(\int_0^{\xi} r^2_{c}(j-s) \ \d s -\xi r^2_{c}(j)\Big)\ \partial_\xi\begin{bmatrix}
        \Gamma^{1,0}(\xi,c)\\
        C^{1,0}(\xi,c)
    \end{bmatrix}[\kappa]\label{eqn:corM3},\end{align}
    and
    \begin{align}\label{eqn:corM4}
    M^{1,0}_{\RomanIV}(\xi,c)=&- \frac{\alpha_0^{-1}(c)}{2c}\mathbb{E} \ \sum_{j\in\Z}\kappa(j)[r^2_{c}(j-\xi)-r^2_{c}(j)]\partial_c\begin{bmatrix}
        \Gamma^{1,0}(\xi,c)\\
        C^{1,0}(\xi,c)
    \end{bmatrix}[\kappa].
    \end{align}
We refer to \Cref{app:correlations} for a fully deterministic representation. As a preparation, we make the following observation regarding these correlations.
\begin{lemma}[See \Cref{app:correlations}]\label{lem:correlations}Assume \Cref{hyp:coeff} and \Cref{hyp:At}. For each $\xi\in\R, c\in(c_-,c_+)$ and $i\in\{I,\ldots,IV\}$, the correlation $M^{1,0}_i(\xi,c)$ can be decomposed as
    \[M_i^{1,0}(\xi,c)=M^{1,0}_{i,\mathrm{per}}(\xi,c)+M^{1,0}_{i,\mathrm{trans}}(\xi,c),\]
    with $M^{1,0}_{i,\mathrm{per}}(\xi,c)$ periodic in $\xi$:
    \[M^{1,0}_{i,\mathrm{per}}(\xi,c)=M^{1,0}_{i,\mathrm{per}}(\xi+d,c),\quad d\in\Z,\]
    and $M^{1,0}_{i,\mathrm{trans}}(\xi,c)$ exponentially decaying in $|\xi|$: there exist constants $C,q>0$ such that
    \[|M^{1,0}_{i,\mathrm{trans}}(\xi,c)|\leq Ce^{-q|\xi|}.\]
\end{lemma}
Upon introducing
\begin{align*}
    M^{1,0}(\xi,c)=&M_{\mathrm{per}}^{1,0}(\xi,c)+M_{\mathrm{trans}}^{1,0}(\xi,c)\\
    =&M^{1,0}_{\RomanI}(\xi,c)+M^{1,0}_{\RomanII}(\xi,c)+M^{1,0}_{\RomanIII}(\xi,c)+M^{1,0}_{\RomanIV}(\xi,c),
\end{align*}
we then obtain the following.
\begin{proposition}\label{prop:expectation}
Assume \Cref{hyp:coeff} and \Cref{hyp:At}. For each $t\geq 0$, we have

\begin{align}\label{eqn:shift}
    \mathbb{E}\begin{bmatrix}
       \dot{\overline{\gamma}}_2(t)\\
        \dot{\overline{c}}_2(t)
    \end{bmatrix}
    =&M^{1,0}(c_*t,c_*)+M^{0,2}(c_*t,c_*)[\overline{R}(\cdot,\cdot,t),\overline{R}(\cdot,\cdot,t)]+M^{1,1}(c_*t,c_*)[\overline{R}(\cdot,\cdot,t)].\end{align}
\end{proposition}

\begin{proof}

By \Cref{hyp:coeff}, the random coefficients satisfy 
\[\mathbb{E}[\kappa(j)]=0 \mand \mathbb{E}[\kappa(i)\kappa(j)]=\delta_{ij},\quad i,j\in\Z.\]
We proceed by computing the expectation of~\eqref{eqn:overline}
using the representation
\[\overline{\eta}(j,t)=\sum_{m\in\Z}\kappa(m)\overline{R}(j,m,t).\]
Via the definitions~\eqref{eqn:corM1}--\eqref{eqn:corM4}, the gradient terms in~\eqref{eqn:overline} satisfy
\begin{align*}\mathbb{E}\ \gamma_{1,\RomanI}(t)\partial_\xi\begin{bmatrix}
        \Gamma^{1,0}(c_*t,c_*)\\
        C^{1,0}(c_*t,c_*)
    \end{bmatrix}[\kappa]=M^{1,0}_{\RomanI}(c_*t,c_*),\\
    \mathbb{E}\ \gamma_{1,\RomanII}(t)\partial_\xi\begin{bmatrix}
        \Gamma^{1,0}(c_*t,c_*)\\
        C^{1,0}(c_*t,c_*)
    \end{bmatrix}[\kappa]=M^{1,0}_{\RomanII}(c_*t,c_*),\\
    \mathbb{E}\ \int_0^t c_1(s) \ \d s \partial_\xi\begin{bmatrix}
        \Gamma^{1,0}(c_*t,c_*)\\
        C^{1,0}(c_*t,c_*)
    \end{bmatrix}[\kappa]=M^{1,0}_{\RomanIII}(c_*t,c_*),\end{align*}
    and
    \begin{align*}
    \mathbb{E} \ c_1(t)\partial_c\begin{bmatrix}
        \Gamma^{1,0}(c_*t,c_*)\\
        C^{1,0}(c_*t,c_*)
    \end{bmatrix}[\kappa]=M^{1,0}_{\RomanIV}(c_*t,c_*) .
    \end{align*}
Here, we recall that the components in $\gamma_{1,\RomanI}(t)$ and $\gamma_{1,\RomanII}(t)$ sum to $\gamma_1(t)$, and are defined in~\eqref{eqn:Y1} and~\eqref{eqn:Y2}, respectively. We proceed with the $\mathcal{O}(\overline{\eta}_r^2)$ term in~\eqref{eqn:overline}. The bilinearity of $\Gamma^{0,2}(c_*t,c_*)$ and $C^{0,2}(c_*t,c_*)$ implies
\begin{align*}
     \mathbb{E}\begin{bmatrix}
        \Gamma^{0,2}(c_*t,c_*)\\
        C^{0,2}(c_*t,c_*)
    \end{bmatrix}[\overline{\eta}_{r}(t),\overline{\eta}_{r}(t)]
     =&\mathbb{E}\begin{bmatrix}
        \Gamma^{0,2}(c_*t,c_*)\\
        C^{0,2}(c_*t,c_*)
    \end{bmatrix}\Big[\sum_{m\in\Z}\kappa(m)\overline{R}_{r}(\cdot,m,t),\sum_{m'\in\Z}\kappa(m')\overline{R}_{r}(\cdot,m',t)\Big]\\
     =&\sum_{m\in\Z}\sum_{m'\in\Z}\mathbb{E}[\kappa(m)\kappa(m')]\begin{bmatrix}
        \Gamma^{0,2}(c_*t,c_*)\\
        C^{0,2}(c_*t,c_*)
    \end{bmatrix}[\overline{R}_{r}(\cdot,m,t),\overline{R}_{r}(\cdot,m',t)]\\
     =&M^{0,2}(c_*t,c_*)[\overline{R}(\cdot,\cdot,t),\overline{R}(\cdot,\cdot,t)].
     \end{align*}
Similarly, 
\begin{align*}\mathbb{E}\begin{bmatrix}
         \Gamma^{1,1}(c_*t,c_*)\\
         C^{1,1}(c_*t,c_*)
     \end{bmatrix}[\kappa,\overline{\eta}(t)]=&\mathbb{E}\begin{bmatrix}
         \Gamma^{1,1}(c_*t,c_*)\\
         C^{1,1}(c_*t,c_*)
     \end{bmatrix}\Big[\sum_{m\in\Z}\kappa(m)\delta(\cdot-m),\sum_{m'\in\Z}\kappa(m')\overline{R}(\cdot,m',t)\Big]\\
     =&\sum_{m\in\Z}\sum_{m'\in\Z}\mathbb{E}[\kappa(m)\kappa(m')]\begin{bmatrix}
         \Gamma^{1,1}(c_*t,c_*)\\
         C^{1,1}(c_*t,c_*)
     \end{bmatrix}\Big[\delta_m,\overline{R}(\cdot,m',t)\Big]\\
     =& M^{1,1}(c_*t,c_*)[\overline{R}(\cdot,\cdot,t)]. \qedhere\end{align*}
\end{proof}
As a consequence of \Cref{prop:Rinf}---the convergence of $\overline{R}$ to $\overline{R}^\infty$ (defined in~\eqref{eqn:Rinf}) in the weighted space $\ell^2_a(\Z^2;\R^2)$---implies that the amplitude attenuation $\mathbb{E}[\dot{\overline{c}}_2(t)]$ also converges in a periodic sense.
\begin{cor}\label{cor:limit}Assume \Cref{hyp:coeff} and \Cref{hyp:At}. For each $p\in[0,1)$, we have
    \begin{align*}
    \lim_{n\to \infty}\mathbb{E}\begin{bmatrix}
        \dot{\overline{\gamma}}_2\big((n+p)/c_*\big)\\
        \dot{\overline{c}}_2\big((n+p)/c_*\big)
    \end{bmatrix}=&M_{\mathrm{per}}^{1,0}(p,c_*)+M^{0,2}(p,c_*)[\overline{R}^\infty(\cdot,\cdot,p),\overline{R}^\infty(\cdot,\cdot,p)]+M^{1,1}(p,c_*)[\overline{R}^\infty(\cdot,\cdot,p)].\end{align*}
    The convergence holds with the same exponential rates as \Cref{prop:Rinf} and \Cref{lem:correlations}, which is independent of $\sigma$.
    \end{cor}
\begin{proof}
    Using the shift invariance
    \[M^{1,1}(c_*t,c_*)[\alpha]=M^{1,1}(c_*t-j,c_*)[\alpha(\cdot+j,\cdot+j)], \quad j\in\Z,\]
    and
    \[M^{0,2}(c_*t,c_*)[\alpha,\beta]=M^{0,2}(c_*t-j,c_*)[\alpha(\cdot+j,\cdot+j),\beta(\cdot+j,\cdot+j)], \quad j\in\Z,\]
    we rewrite~\eqref{eqn:shift} as
    \begin{align*}
    \mathbb{E}\begin{bmatrix}
        \dot{\overline{\gamma}}_2(t)\\
        \dot{\overline{c}}_2(t)
    \end{bmatrix}=&M^{1,0}(c_*t,c_*)+M^{0,2}\big(c_*t-\lfloor c_*t\rfloor,c_*\big)\big[\overline{R}(\cdot+\lfloor c_*t\rfloor,\cdot+\lfloor c_*t\rfloor,t),\overline{R}(\cdot+\lfloor c_*t\rfloor,\cdot+\lfloor c_*t\rfloor,t)\big]\\
    &+M^{1,1}\big(c_*t-\lfloor c_*t\rfloor,c_*\big)\big[\overline{R}(\cdot+\lfloor c_*t\rfloor,\cdot+\lfloor c_*t\rfloor,t)\big].\end{align*}
   Thus, for $p\in [0,1)$ and \[t_n=(n+p)/c_*, \quad n\in\N,\]
we have
 \begin{align*}
    \mathbb{E}\begin{bmatrix}
        \dot{\overline{\gamma}}_2(t_n)\\
        \dot{\overline{c}}_2(t_n) \end{bmatrix}=&
M^{1,0}(n+p,c_*)+M^{0,2}(p,c_*)\big[\overline{R}(\cdot+n,\cdot+n,t_n),\overline{R}(\cdot+n,\cdot+n,t_n)\big]\\&+M^{1,1}(p,c_*)\big[\overline{R}(\cdot+n,\cdot+n,t_n)\big].\end{align*}
    The result now follows from the convergence 
    \[M^{1,0}(n+p,c_*)=M_{\mathrm{per}}^{1,0}(p,c_*)+M_{\mathrm{trans}}^{1,0}(n+p,c_*)\to M_{\mathrm{per}}^{1,0}(p,c_*),\]
    through \Cref{lem:correlations}, and the convergence
    of \[\overline{R}(\cdot+n,\cdot+n,t_n)\to \overline{R}^\infty(\cdot,\cdot,p)\] 
    in $\ell^2_a(\Z^2;\R^2)$ (\Cref{prop:Rinf}). Indeed, the linear operator $ M^{1,1}(p,c_*)$ and bilinear map $ M^{0,2}(p,c_*)$ are bounded on $\ell^2_a(\Z^2;\R^2)$ (\Cref{app:continuity}).
   
\end{proof}
Since $\mathbb{E}[\dot{\overline{c}}_2(t)]$ converges to a $1/c_*$-periodic function, the quantity which effectively captures the decay-rate of the solitary wave amplitude is the second component of
\begin{align}\begin{bmatrix}
    \overline{\mathcal{Q}}_{\gamma}(c_*)\\
    \overline{\mathcal{Q}}_{c}(c_*)
\end{bmatrix}:=c_*^{-1}\int_0^{1}M_{\mathrm{per}}^{1,0}(p,c_*)+M^{0,2}(p,c_*)[\overline{R}^\infty(\cdot,\cdot,p),\overline{R}^\infty(\cdot,\cdot,p)]+M^{1,1}(p,c_*)[\overline{R}^\infty(\cdot,\cdot,p)]\ \d p.\label{eqn:rate}\end{align}
In particular, we write $\mathcal{Q}_{c}(c_*)$ for the version of $\overline{\mathcal{Q}}_{c}(c_*)$ that arises by taking $\mathcal{A}_{c_*t,c_*}=\mathcal{L}_{c_*t,c_*}$ in~\eqref{eqn:etaline}. Similarly, we write ${\mathcal{Q}}^{\mathrm{h}}_{c}(c_*)$ for the version of $\overline{\mathcal{Q}}_{c}(c_*)$ corresponding to the constant coefficient reduction $\mathcal{A}_{c_*t,c_*}=I$ in~\eqref{eqn:etaline}. Associated to these quantities, we can define the limiting ODEs
\begin{align}\dot{c}_{\mathrm{lim}}(\tau)=\mathcal{Q}_c({c}_{\mathrm{lim}}(\tau)),\quad \ \mathrm{with} \quad c_{\mathrm{lim}}(0)=c_*,\label{eqn:clim}\end{align}
and
\begin{align}\dot{c}^{\mathrm{h}}_{\mathrm{lim}}(\tau)=\mathcal{Q}^{\mathrm{h}}_c({c}^{\mathrm{h}}_{\mathrm{lim}}(\tau)),\quad \ \mathrm{with} \quad c^{\mathrm{h}}_{\mathrm{lim}}(0)=c_*,\label{eqn:limh}\end{align}
which capture the substantial deviations of $c$ from $c_*$ in the slow time variable $\tau=\sigma^2t$.

\Cref{fig:pdependence} shows that the dependence of the integrand in~\eqref{eqn:rate} on $p$ is very minimal.
\begin{figure}[h]
    \centering
    \includegraphics[width=0.5\linewidth]{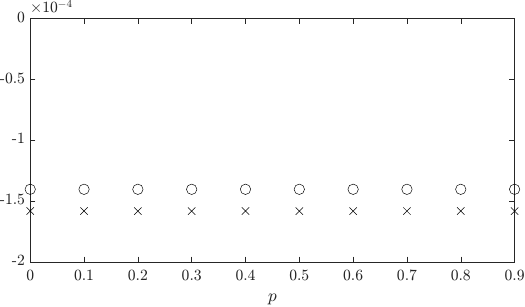}
    \caption{Asymptotic mean attenuation $\lim_{n\to\infty}\mathbb{E}[\dot{c}_2((n+p)/c_*)]$ (marked with $\circ$) and $\lim_{n\to\infty}\mathbb{E}[\dot{c}_2^{\mathrm{h}}((n+p)/c_*)]$ (marked with $\times$), as given by \Cref{cor:limit}, for $p\in[0,1)$.}
    \label{fig:pdependence}
\end{figure}
The right panel of \Cref{fig:rate_combined} displays the nonlinear dependence on the decay rates $\mathcal{Q}_{c}(c)$ and $\mathcal{Q}^{\mathrm{h}}_{c}(c)$ on the amplitude $c$. This sheds light on the nature of the amplitude decay through~\eqref{eqn:clim}. A polynomial fit on the datapoints in \Cref{fig:rate_combined} suggests a relation of the form $\mathcal{Q}_{c}(c)\sim -(c-1)^{2}$, corresponding to a polynomial decay $c(t)-1\sim \frac{1}{\sigma^2 t}$. \Cref{fig:sigmastangents+ODE} confirms that the limiting ODEs~\eqref{eqn:clim} and~\eqref{eqn:limh} track the amplitude decay remarkably well over long timescales.
\begin{figure}[H]
    \centering
    \begin{subfigure}{0.48\linewidth}
        \centering
        \includegraphics[width=\linewidth]{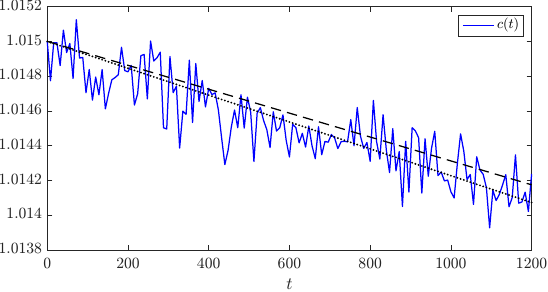}
    \end{subfigure}\hfill
    \begin{subfigure}{0.48\linewidth}
        \centering
        \includegraphics[width=\linewidth]{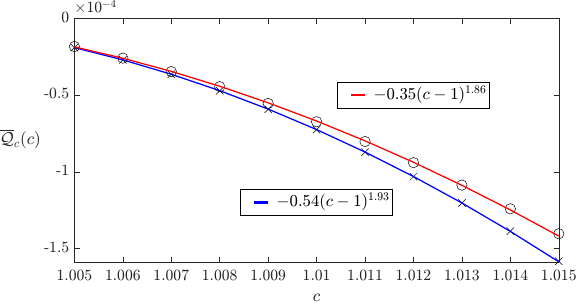}
    \end{subfigure}
    \caption{The left panel shows a realization of $c(t)$ and the asymptotic means $c_*+\mathcal{Q}_c(c_*)\sigma^2t$ (dashed) and $c_*+\mathcal{Q}^{\mathrm{h}}_c(c_*)\sigma^2t$ (dotted), via a numerical implementation of~\eqref{eqn:rate}. For this realization, $\sigma=0.07$, $c_*=1.015$ and $\kappa(i)\in \text{Unif}([-\sqrt{3},\sqrt{3}])$. The right panel displays numerical evaluations of $\mathcal{Q}_{c}(c)$  (marked with $\circ$) and $\mathcal{Q}^{\mathrm{h}}_{c}(c)$ (marked with $\times$), for various values of $c$; see \Cref{app:scheme}. The solid lines represent fitted power laws. } 
    \label{fig:rate_combined}
\end{figure}
\begin{figure}[h]
    \centering
    \begin{subfigure}{0.48\linewidth}
        \includegraphics[width=\linewidth]{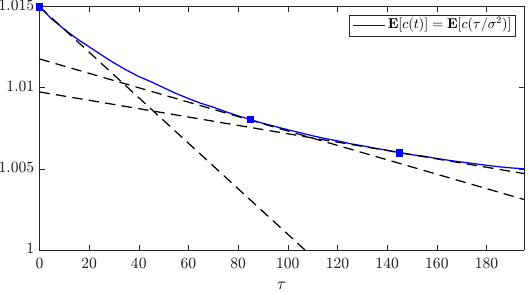}
        
    \end{subfigure}
    \hfill
    \begin{subfigure}{0.48\linewidth}
        \includegraphics[width=\linewidth]{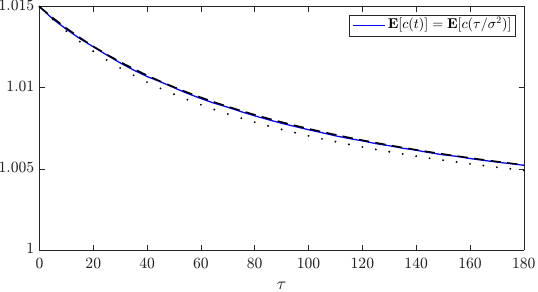}
        
    \end{subfigure}
    \caption{\Cref{fig:sigmas} revisited. The solid curves represent the sample mean $\mathbf{E}[c(t)]$, with $t=\tau/\sigma^2$, computed over 100 realizations, for $\sigma=0.1$, $c_*=1.015$, and $\kappa(i)$ drawn from a uniform distribution on $[-\sqrt{3},\sqrt{3}]$. The dashed lines in the left panel have slope $\mathcal{Q}_c(c)$ for $c \in \{1.015, 1.008, 1.006\}$, and are vertically shifted to intersect the $\mathbf{E}[c(t)]$ curve at the point where $\mathbf{E}[c(t)] = c$. The dashed and dotted curves in the right panel represent $c_{\mathrm{lim}}(\tau)$ and $c^{\mathrm{h}}_{\mathrm{lim}}(\tau)$, respectively, with $c_*=1.015$. Observe that the dashed curve coincides with the solid curve. }
    \label{fig:sigmastangents+ODE}
\end{figure}

\section{Outlook}\label{sec:outlook}
While the second-order expansion terms identified in \Cref{sec:reduced} reveal the $\mathcal{O}(\sigma^2)$ attenuation affecting the propagation of solitary waves in~\eqref{eqn:FPU}, our future goal is to rigorously justify its effect on the exact amplitude $c(t)$. This requires, presumably among other things, asserting that the solitary waves remain coherent as they slowly decay in amplitude. That is, $\eta(t)$ must stay small in some (weighted) norm, at least for times proportional to $\sigma^{-2}$. The main difficulty towards establishing this via the linear stability tool \Cref{prop:linearstability} is to control the nonlinearity $\eta_r^2$ present in~\eqref{eqn:etadot}. Classically, this is controlled by estimating the weighted norm
\[\|e^{a(\cdot-\xi(t))}\eta^2_r(t)\|_{\ell^2(\Z;\R)}\leq \|\eta_r(t)\|_{\ell^\infty(\Z;\R)}\|e^{a(\cdot-\xi(t))}\eta_r(t)\|_{\ell^2(\Z;\R)}\leq \|\eta_r(t)\|_{\ell^2(\Z;\R)}\|e^{a(\cdot-\xi(t))}\eta_r(t)\|_{\ell^2(\Z;\R)},\]
where the unweighted norm $\|\eta_r(t)\|_{\ell^2(\Z;\R)}$ can in turn be understood via energy methods. However, since the tail $\eta_r(t)$ expands behind the soliton, its squared $\ell^2(\Z;\R)$-norm grows linearly with time. Indeed, the amplitude attenuation goes hand-in-hand with the formation of a tail. We anticipate that future work will focus on controlling the supremum norm $\|\eta_r(t)\|_{\ell^\infty(\Z;\R)}$. The right panel of \Cref{fig:tail} suggests that the linear tail $\overline{\eta}(t)=\eta_1^{\mathrm{h}}(t)$ satisfies uniform bounds in space and time. The challenge will be to develop a pointwise analysis for the nonlinear tail $\eta(t)$.

We also briefly discussed the near-sonic behavior of the amplitude attenuation, i.e. the nature of the decay as $c\downarrow 1$. Future work can analyze the derivatives of $\mathcal{Q}_{c}(c)$ (\Cref{sec:attenuation}) at $c=1$ in more detail, to identify the lowest non-zero coefficient in an expansion
\[\mathcal{Q}_{c}(c)=k_1(c-1)+k_2(c-1)^2+\ldots \ .\]
This could confirm if the decay is of the form \[c(t)-1\sim\frac{1}{\sigma^2 t}.\]

\appendix

\section{Statistics of leading-order phase and amplitude}\label{app:statistics}
Below, we prove elementary statistical properties of the leading-order phase and amplitude contributions $(\gamma_{1}(t),c_{1}(t))$ in~\eqref{eqn:expgamma} and~\eqref{eqn:expc}.
\begin{proof}[Proof of \Cref{prop:variance}]
    From the explicit solutions~\eqref{eqn:c1explicit},~\eqref{eqn:Y1},~\eqref{eqn:Y2} and the assumption 
    \[\mathbb{E}[\kappa(i)]=0,\quad \mathbb{E}[\kappa^2(i)]=1,\quad i\in\Z\]
    we immediately obtain
    \[\mathbb{E}[c_1(t)]=c_*,\quad \mathbb{E}[\gamma_1(t)]=0,\quad t\geq 0.\]
    Turning to the covariance, we compute
   \begin{align*}\mathbb{E}[c_1(t)c_1(s)]=&\frac{\alpha_0^{-2}(c_*)}{4c_*^2}\sum_{j\in\Z}\sum_{j'\in\Z}\mathbb{E}[\kappa(j)\kappa(j')](r^2_{c_*}(j-c_*t)-r^2_{c_*}(j))(r^2_{c_*}(j'-c_*s)-r^2_{c_*}(j')).
   \end{align*}
   Since the $\kappa(j)$'s are independent, this reduces to    \begin{align*}\mathbb{E}[c_1(t)c_1(s)]=&\frac{\alpha_0^{-2}(c_*)}{4c_*^2}\sum_{j\in \Z}(r_{c_*}^2(j)-r^2_{c_*}(j-c_*t))(r_{c_*}^2(j)-r^2_{c_*}(j-c_*s))\end{align*}
   and the variance bound~\eqref{eqn:variancebound} then follows directly. Moving on to the variance of $\gamma_1(t)= \gamma_{1,\RomanI}(t)+ \gamma_{1,\RomanII}(t)$, we obtain in an analogous way from~\eqref{eqn:Y1} that
   \[\mathbb{E}[\gamma^2_{1,\RomanI}(t)]\leq \frac{\alpha^{-4}_0(c_*)\alpha^2_1(c_*)}{c^2_*}\|r_{c_*}\|_{\ell^4(\Z;\R)}^4.\]
   Furthermore, we obtain from~\eqref{eqn:Y2} that
     \begin{align*}\mathbb{E}[\gamma^2_{1,\RomanII}(t)]=& \frac{\alpha^{-2}_0(c_*)}{4}\sum_{j\in\Z}\Big(\int_0^t\partial_c r^2_{c_*}(j-c_*s)\d s\Big)^2\\
     =&\frac{\alpha^{-2}_0(c_*)}{4}\sum_{j\in\Z}\Big(\int_{j-c_*t}^j\partial_c r^2_{c_*}(y)\d y\Big)^2\\
     \leq&C\frac{\alpha^{-2}_0(c_*)}{4}\sum_{j\in\Z}\Big(\int_{j-c_*t}^je^{-\beta|y|}\d y\Big)^2.
     \end{align*}
     Here, we used that $\partial_cr_{c_*}^2$ is exponentially localized \cite[Proposition 5.5]{FPI}:
     \[\partial_cr_{c_*}^2(x)\leq Ce^{-\beta|x|}\]
     for some $C,\beta>0$. Note that
    \[
\int_{j - c_* t}^j e^{-\beta |y|} \, \d y =
\begin{cases}
\beta^{-1} \left( e^{\beta j} - e^{\beta(j - c_* t)} \right), & \text{if } j \leq 0, \\[1ex]
\beta^{-1} \left( 2 - e^{\beta(j - c_* t)} - e^{-\beta j} \right), & \text{if } 0 < j \leq c_* t, \\[1ex]
\beta^{-1} \left( e^{-\beta(j - c_* t)} - e^{-\beta j} \right), & \text{if } j > c_* t.
\end{cases}
\]
Hence,
     \begin{align*}\sum_{j\in\Z}\Big(\int_{j-c_*t}^je^{-\beta|y|}\d y\Big)^2\leq& 2\beta^{-2}\sum_{j\leq 0}e^{2\beta j}+4c_*t+2\beta^{-2}\sum_{j>ct}e^{-2\beta j}\\
     \leq& \frac{4\beta^{-2}}{1-e^{-2\beta}}+4c_*t,
     \end{align*}
     and
     \[\mathbb{E}[\gamma^2_{1,\RomanII}(t)]\leq C\frac{\alpha^{-2}_0(c_*)}{4}\Big(\frac{4\beta^{-2}}{1-e^{-2\beta}}+4c_*t\Big).\]
     The result now follows by estimating \[\mathbb{E}[\gamma_1^2(t)]=\mathbb{E}[\gamma^2_{1,\RomanI}(t)]+2\mathbb{E}[\gamma_{1,\RomanI}(t)\gamma_{1,\RomanII}(t)]+\mathbb{E}[\gamma^2_{1,\RomanII}(t)]\leq 2\mathbb{E}[\gamma^2_{1,\RomanI}(t)]+\mathbb{E}[\gamma^2_{1,\RomanII}(t)].\qedhere\]
\end{proof}

\section{Kernel representations}\label{app:greens}
Here, we derive the explicit kernel formula~\eqref{eqn:kernel} and prove \Cref{lem:greens} regarding the Green's function associated to general time/shift invariant linear operators. 
\begin{lemma}\label{lem:discwave}
Let $u_0,v_0\in\ell^2(\Z;\R)$. The linear system
\begin{align*}
   \begin{bmatrix}
        \dot{u}(t,j)\\
        \dot{v}(t,j)
    \end{bmatrix}=\mathcal{J}\begin{bmatrix}
        {u}(t,j)\\
        {v}(t,j)
    \end{bmatrix},\quad \text{with} \quad  \begin{bmatrix}
        {u}(0,j)\\
        {v}(0,j)
    \end{bmatrix}=\begin{bmatrix}
        {u}_0(j)\\
        {v}_0(j)
    \end{bmatrix}
\end{align*}
is solved by 
\begin{align*}
   \begin{bmatrix}
        u(t,j)\\
        v(t,j)
    \end{bmatrix}=\sum_{k\in\Z}\Phi(j-k,t)\begin{bmatrix}
        u_0(k)
\\
v_0(k)\end{bmatrix},\quad j\in\Z,\ t\in \R.
\end{align*}
\end{lemma}
\begin{proof}
We note that the Bessel functions satisfy $J_n(0)=0$ for all $n\in \Z\setminus\{0\}$ and meet the recurrence relation
\begin{align}2J_n'(x)=J_{n-1}(x)-J_{n+1}(x),\label{eqn:rec}\end{align}
see for instance \cite[\S 2.12]{bessel}. Consider initial conditions given by the basis vectors $(u_0,\ v_0)^\top=(\delta_0,\ 0)^\top$ and $(u_0,\ v_0)^\top=(0,\ \delta_0)^\top$. These give rise to the fundamental solutions \begin{align*}
    \phi(j,t)=&\begin{bmatrix}
        J_{2j}(2t)\\
        -J_{2j-1}(2t)
    \end{bmatrix}\quad \text{and} \quad 
    \psi(j,t)=\begin{bmatrix}
        -J_{2j+1}(2t)\\
        J_{2j}(2t)
    \end{bmatrix}
\end{align*}
for the discrete wave operator $\mathcal{J}$. Indeed, through the recurrence relation~\eqref{eqn:rec} one verifies that 
\begin{align*}\dot{\phi}(j,t)=\mathcal{J}\phi(j,t),\quad \text{with}\quad \phi(j,0)=\begin{bmatrix}\delta_0(j)\\0\end{bmatrix}\end{align*}
and
\begin{align*}\dot{\psi}(j,t)=\mathcal{J}\psi(j,t),\quad \text{with}\quad \psi(j,0)=\begin{bmatrix}0\\\delta_0(j)\end{bmatrix}.\end{align*}
For arbitrary initial conditions $(u_0,\ v_0)^\top\in \ell^2(\Z;\R^2)$, we obtain the solution~\eqref{eqn:linearwave} through a convolution with the fundamental solutions.
\end{proof}
We then move on to the proof of \Cref{lem:greens}. We refer to \cite[Theorem 4.2]{benzoni} for a constructive proof in the context of the discretized conservation laws. 
\begin{proof}[Proof of \Cref{lem:greens}]
    In terms of the shift-operation $S\eta=\eta(\cdot+1)$,~\eqref{eqn:invariance} reads
\[S^d\mathcal{A}_{c_*t,c_*}=\mathcal{A}_{c_*t-d,c_*}S^d,\quad d\in\Z.\]
We claim that this translates to the following property of the associated evolution family:
\begin{align}S^dU_{c_*}(t,s)=U_{c_*}(t-d/c_*,s-d/c_*)S^d,\quad d\in\Z.\label{eqn:evolutionshift}\end{align}
We prove the case $d=1$ by first differentiating the left-hand side:
\[\partial_t (SU_{c_*}(t,s))=S\mathcal{A}_{c_*t,c_*}U_{c_*}(t,s)=\mathcal{A}_{c_*t-1,c_*}SU_{c_*}(t,s).\]
On the other hand:
\[\partial_t(U_{c_*}(t-1/c_*,s-1/c_*)S)=\mathcal{A}_{c_*t-1,c_*}U_{c_*}(t-1/c_*,s-1/c_*)S.\]
Hence, both terms satisfy the (operator-valued) ODE
\[\dot{x}(t)=\mathcal{A}_{c_*t-1,c_*}x(t) \quad \text{with} \quad x(s)=S.\]
Next, we observe that for each $t\geq s$, the bounded operator $U_{c_*}(t,s)$ admits a kernel representation
\[U_{c_*}(t,s)\begin{bmatrix}
    \alpha \\ \beta
\end{bmatrix}(j)=\sum_k \Phi_{c_*}(t,s;j,k)\begin{bmatrix}
    \alpha(k) \\ \beta(k)
\end{bmatrix}.\]
Identity~\eqref{eqn:evolutionshift} has immediate consequences for the structure of the kernel $\Phi_{c_*}$:
\begin{align*}
    \sum_{k}\Phi_{c_*}(t,s;j+d,k)\begin{bmatrix}
        \alpha(k)
\\
\beta(k)\end{bmatrix}=&\Big(U_{c_*}(t-d/c_*,s-d/c_*)\begin{bmatrix}
        \alpha(\cdot+d)
\\
\beta(\cdot+d)\end{bmatrix}\Big)(j)\\
=&\sum_{k}\Phi_{c_*}(t-d/c_*,s-d/c_*;j,k)\begin{bmatrix}
        \alpha(k+d)
\\
\beta(k+d)\end{bmatrix}\\
=&\sum_{k'}\Phi_{c_*}(t-d/c_*,s-d/c_*;j,k'-d)\begin{bmatrix}
        \alpha(k')
\\
\beta(k')\end{bmatrix}
\end{align*}
and consequently 
\begin{align}\Phi_{c_*}(t,s;j,k)=\Phi_{c_*}(t-d/c_*,s-d/c_*;j-d,k-d), \quad d\in\Z.\label{eqn:greenshift}\end{align}
Through a basis transformation of the arguments $(t,s;j,k)$, we obtain a kernel $\tilde{\mathcal{G}}_{c_*}$ which satisfies
\[\Phi_{c_*}(t,s;j,k)=\tilde{\mathcal{G}}_{c_*}(t-s,j-c_*t,k-c_*s,k).\]
In view of~\eqref{eqn:greenshift}, $\tilde{\mathcal{G}}$ is constant in its last variable. It follows that the kernel only depends on three effective variables:
\[\Phi_{c_*}(t,s;j,k)=\mathcal{G}_{c_*}(t-s,j-c_*t,k-c_*s). \qedhere\]
\end{proof}

\section{Continuity of (bi)linear maps}\label{app:continuity}
Here, we prove that the linear map $M^{1,1}(\xi,c):\ell^2_a(\Z^2;\R^2)\to\R^2$ defined via
\begin{align*}M^{1,1}(\xi,c)[\alpha]=\sum_{m\in\Z}\begin{bmatrix}
        \Gamma^{1,1}(\xi,c)\\
        C^{1,1}(\xi,c)
    \end{bmatrix}[\delta_m,\alpha(\cdot,m)],
\end{align*}
and the bilinear map $M^{0,2}(\xi,c):\ell^2_a(\Z^2;\R^2)\times \ell^2_a(\Z^2;\R^2)\to\R^2$ defined by
\begin{align*}M^{0,2}(\xi,c)[\alpha,\beta]=\sum_{m\in\Z}\begin{bmatrix}
        \Gamma^{0,2}(\xi,c)\\
        C^{0,2}(\xi,c)
    \end{bmatrix}[\alpha_{r}(\cdot,m),\beta_{r}(\cdot,m)],
\end{align*}
used in \Cref{sec:attenuation}, are continuous on exponentially weighted spaces
\[
\ell^2_a(\mathbb{Z}^2;\mathbb{R}^2)
:= \left\{ f:\mathbb{Z}^2\to\mathbb{R}^2 \;\middle|\; \|f\|_{\ell^2_a(\mathbb{Z}^2;\mathbb{R}^2)}<\infty \right\},\]
with norm
\[
\|f\|_{\ell^2_a(\mathbb{Z}^2;\mathbb{R}^2)}^2
= \sum_{m\in\mathbb{Z}} \|e^{a\cdot} f(\cdot,m)\|_{\ell^2(\Z;\R^2)}^2 .
\]
We recall that the weight $a>0$ was introduced in \Cref{prop:linearstability}, and that $\Gamma^{1,1}, \Gamma^{0,2}, C^{1,1}$ and $C^{0,2}$ are expansion terms of the modulation system, as introduced in \Cref{sec:modulation}.
\begin{lemma}\label{lem:continuityML}
   For each $p\in[0,1)$ and $c\in(c_-,c_+)$, the linear operator $M^{1,1}(p,c)$ is bounded from $\ell^2_a(\mathbb{Z}^2;\mathbb{R}^2)$ to $\R^2$.
\end{lemma}
\begin{proof}
Let $p\in[0,1), c\in(c_-,c_+)$, and $\alpha=(\alpha_r,\ \alpha_p)^\top\in \ell^2_a(\mathbb{Z}^2;\mathbb{R}^2)$. 
        From~\eqref{eqn:MLdef}, we get
 \begin{align*}
        \Big|M^{1,1}(p,c)[\alpha]\Big|
        \leq& 
        \sum_{m\in\Z}\Big|\begin{bmatrix}
        \Gamma^{1,1}(\xi,c)\\
        C^{1,1}(\xi,c)
    \end{bmatrix}[\delta_m,\alpha(\cdot,m)]\Big|
        \leq S_1+S_2,
    \end{align*}
    with
    \[S_1=\sum_{m\in\Z}\Big|\begin{bmatrix}
         \Gamma^{0,2}(p,c)\\
         C^{0,2}(p,c)
     \end{bmatrix}[\delta_m,\alpha_r(\cdot,m)]\Big|,\]
     and
   \[S_2=\sum_{m\in\Z}\Big|A^{-1}(c)B(p,c,\alpha(\cdot,m))\begin{bmatrix}
         \Gamma^{1,0}(p,c)\\
         C^{1,0}(p,c)
     \end{bmatrix}[\delta_m]\Big|.\]  
We estimate
    \begin{align*}
\Big|\begin{bmatrix}
         \Gamma^{0,2}(p,c)\\
         C^{0,2}(p,c)
     \end{bmatrix}[\delta_m,\alpha_r(\cdot,m)]\Big|\leq& C\Big(| \partial_\xi r_c(m-p)\alpha_r(m,m)|+| \partial_c r_c(m-p)\alpha_r(m,m)|\Big),
    \end{align*}
so that
\begin{align*}
    S_1\leq& C\sum_{m\in\Z} \big(|\partial_\xi r_c(m-p)+|\partial_c r_c(m-p)|\big) |\alpha_r(m,m)|\\
    \leq& C\Big(\sum_{m\in\Z}e^{2am}|\alpha_r(m,m)|^2\Big)^{1/2} \Big(\sum_{m\in\Z} e^{-2am}\big(|\partial_\xi r_c(m-p)+|\partial_c r_c(m-p)|\big)^2\Big)^{1/2}\\
    \leq& \tilde{C}\Big(\sum_{m\in\Z}\|e^{a\cdot}\alpha_r(\cdot,m)\|_{\ell^2(\Z;\R)}^2\Big)^{1/2}\leq \tilde{C}\|\alpha\|_{\ell^2_a(\mathbb{Z}^2;\mathbb{R}^2)}.
\end{align*}
We proceed with 
\begin{align*}
    S_2\leq& \Big(\sum_{m\in\Z}\Big|A^{-1}(c)B(p,c,\alpha(\cdot,m))\Big|^2\Big)^{1/2}  \Big(\sum_{m\in\Z}\Big|\begin{bmatrix}
         \Gamma^{1,0}(p,c)\\
         C^{1,0}(p,c)
     \end{bmatrix}[\delta_m]\Big|^2\Big)^{1/2},
\end{align*}
where
\begin{align*}
    \sum_{m\in\Z}\Big|\begin{bmatrix}
         \Gamma^{1,0}(p,c)\\
         C^{1,0}(p,c)
     \end{bmatrix}[\delta_m]\Big|^2\leq C \sum_{m\in\Z}\Big(|\partial_\xi r_c^2(m-p)|+|\partial_c r_c^2(m-p)|\Big)^2<\infty.
\end{align*}
The result then follows by estimating
\begin{align*}
    \Big|A^{-1}(c)B(p,c,\alpha(\cdot,m))\Big|\leq CW\|e^{a\cdot}(\alpha(\cdot,m))\|_{\ell^2(\Z;\R^2)},
\end{align*}
with 
\begin{align*}
    W:=\|e^{-a\cdot} \mathcal{J}^{-1}\partial_{\xi\xi}^2\phi_c(\cdot-p)\|_{\ell^2(\Z;\R^2)}+2\|e^{-a\cdot} \mathcal{J}^{-1}\partial_{\xi c}^2\phi_c(\cdot-p)\|_{\ell^2(\Z;\R^2)}+\|e^{-a\cdot} \mathcal{J}^{-1}\partial_{cc}^2\phi_c(\cdot-p)\|_{\ell^2(\Z;\R^2)}.
\end{align*}
In particular
\[
    \sum_{m\in\Z}\Big|A^{-1}(c)B(p,c,\alpha(\cdot,m))\Big|^2\leq C^2W^2\sum_{m\in\Z}\|e^{a\cdot}(\alpha(\cdot,m))\|_{\ell^2(\Z;\R^2)}^2=C^2W^2\|\alpha\|^2_{\ell^2_a(\mathbb{Z}^2;\mathbb{R}^2)} .
\qedhere \]
\end{proof}
We proceed with the bilinear map $M^{0,2}$.
\begin{lemma}\label{lem:continuityMQ}
   For each $p\in[0,1)$ and $c\in(c_-,c_+)$, the bilinear form $M^{0,2}(p,c)$ is bounded from $\ell^2_a(\mathbb{Z}^2;\mathbb{R}^2)\times \ell^2_a(\mathbb{Z}^2;\mathbb{R}^2)$ to $\R^2$. In particular, there exists a constant $C>0$ such that 
   \[|M^{0,2}(p,c)[\alpha,\beta]|\leq C\|\alpha\|_{\ell^2_a(\mathbb{Z}^2;\mathbb{R}^2)}\|\beta\|_{\ell^2_a(\mathbb{Z}^2;\mathbb{R}^2)},\]
   for all $\alpha,\beta\in \ell^2_a(\Z^2;\R^2).$
\end{lemma}

\begin{proof}
Let $p\in[0,1)$ and $ c\in(c_-,c_+)$. Let furthermore 
\[\alpha=(\alpha_r,\ \alpha_p)^\top,\beta=(\beta_r,\ \beta_p)^\top \in \ell^2_a(\mathbb{Z}^2;\mathbb{R}^2).\] 
From~\eqref{eqn:MQdef}, we estimate
    \[|M^{0,2}(p,c)[\alpha,\beta]|\leq\sum_{m\in\Z} \Big(\Big|\begin{bmatrix}
        \Gamma^{0,2}(p,c)\\
        C^{0,2}(p,c)
    \end{bmatrix}[\alpha_r(\cdot,m),
        \beta_r(\cdot,m)]\Big|\Big).\]
 For all $m\in\Z$, we have
    \begin{align*}
        \Big|\begin{bmatrix}
        \Gamma^{0,2}(p,c)\\
        C^{0,2}(p,c)
    \end{bmatrix}[\alpha_r(\cdot,m),\beta_r(\cdot,m)]\Big|\leq& C|\langle \partial_\xi r_c(\cdot-p),\alpha_r(\cdot,m)\beta_r(\cdot,m)\rangle_{\ell^2(\Z;\R)}|\\
        &+C|\langle \partial_c r_c(\cdot-p),\alpha_r(\cdot,m)\beta_r(\cdot,m)\rangle_{\ell^2(\Z;\R)}|\\
        \leq & C\sup_{j\in\Z}e^{-2aj}\big(| \partial_\xi r_c(j-p)|+|\partial_c r_c(j-p)|\big)\\
        &\times|\langle e^{a\cdot}\alpha_r(\cdot,m),e^{a\cdot}\beta_r(\cdot,m)\rangle_{\ell^2(\Z;\R)}|\\
        \leq & \tilde{C} \| e^{a\cdot}\alpha_r(\cdot,m)\|_{\ell^2(\Z;\R)} \|e^{a\cdot}\beta_r(\cdot,m)\|_{\ell^2(\Z;\R)}, 
    \end{align*}
    so that
 \begin{align*}
|M^{0,2}(p,c)[\alpha,\beta]|
\leq & \tilde{C}\sum_{m\in\Z}\| e^{a\cdot}\alpha_r(\cdot,m)\|_{\ell^2(\Z;\R)} \|e^{a\cdot}\beta_r(\cdot,m)\|_{\ell^2(\Z;\R)}\\
\leq & \tilde{C}\Big(\sum_{m\in\Z}\| e^{a\cdot}\alpha_r(\cdot,m)\|_{\ell^2(\Z;\R)}^2\Big)^{1/2}\Big(\sum_{m\in\Z}\|e^{a\cdot}\beta_r(\cdot,m)\|^2_{\ell^2(\Z;\R)}\Big)^{1/2}\\
\leq & \tilde{C}\|\alpha\|_{\ell^2_a(\mathbb{Z}^2;\mathbb{R}^2)}\|\beta\|_{\ell^2_a(\mathbb{Z}^2;\mathbb{R}^2)}.\qedhere
 \end{align*}   
\end{proof}

\section{Correlations}\label{app:correlations}
Here, we prove \Cref{lem:correlations} regarding the correlations $M^{1,0}_{\RomanI},\ldots,M^{1,0}_{\RomanIV}$ introduced in~\eqref{eqn:corM1}--\eqref{eqn:corM4}. By \Cref{hyp:coeff}, we have
\[\mathbb{E}[\kappa(j)]=0 \mand \mathbb{E}[\kappa(i)\kappa(j)]=\delta_{ij},\quad i,j\in\Z.\]
This allows us to explicitly compute
   \begin{align*} 
    M^{1,0}_{1}(\xi,c)
    =&\frac{\alpha^{-2}_0(c)\alpha_1(c)}{4c}\sum_{j\in\Z}[r^2_{c}(j-\xi)-r^2_{c}(j)]A^{-1}(c)\begin{bmatrix}
        (\partial^2_{\xi\xi} r^2_{c})(j-\xi)\\
        (\partial^2_{\xi c}r^2_{c})(j-\xi)
    \end{bmatrix}, \\
    M^{1,0}_{2}(\xi,c)
    =& -\frac{\alpha^{-1}_0(c)}{4c}\sum_{j\in\Z}\int_0^{\xi}\partial_c r^2_{c}(j-s)\d s A^{-1}(c)\begin{bmatrix}
        (\partial^2_{\xi\xi} r^2_{c})(j-\xi)\\
        (\partial^2_{\xi c}r^2_{c})(j-\xi)
    \end{bmatrix},\\
   M^{1,0}_{\RomanIII}(\xi,c)=&- \frac{\alpha_0^{-1}(c)}{4c^2}\sum_{j\in\Z}\Big(\int_0^{\xi} r^2_{c}(j-s) \ \d s -\xi r^2_{c}(j)\Big)A^{-1}(c)\begin{bmatrix}
        (\partial^2_{\xi\xi} r^2_{c})(j-\xi)\\
        (\partial^2_{\xi c}r^2_{c})(j-\xi)
    \end{bmatrix},
    \end{align*}
    together with
    \begin{align*}
        M^{1,0}_{\RomanIV}(\xi,c)=&\frac{\alpha_0^{-1}(c)}{4c}\sum_{j\in\Z}[r^2_{c}(j-\xi)-r^2_{c}(j)]\Big(\partial_cA^{-1}(c)\begin{bmatrix}
        (\partial_\xi r^2_{c})(j-\xi)\\
        (\partial_cr^2_{c})(j-\xi)
    \end{bmatrix}+A^{-1}(c)\begin{bmatrix}
        (\partial^2_{c\xi} r^2_{c})(j-\xi)\\
        (\partial^2_{cc}r^2_{c})(j-\xi)
    \end{bmatrix}\Big).
    \end{align*}
\begin{proof}[Proof of \Cref{lem:correlations}]

    We prove that \[M_{\RomanIII}^{1,0}(\xi,c)=M^{1,0}_{\RomanIII,\mathrm{per}}(\xi,c)+M^{1,0}_{\RomanIII,\mathrm{trans}}(\xi,c),\]
    with $M^{1,0}_{\RomanIII,\mathrm{per}}(\xi,c)$ 1-periodic in $\xi$
    and $M^{1,0}_{\RomanIII,\mathrm{trans}}(\xi,c)$ exponentially decaying in $|\xi|$. The proof for the remaining mappings is then analogous.
    Decomposing \[\R\ni\xi=n+p, \quad n\in\Z,\ p\in[0,1),\] we observe that
    \begin{align*}
        M^{1,0}_{\RomanIII}(\xi,c)=&- \frac{\alpha_0^{-1}(c)}{4c^2}\sum_{j\in\Z}\Big(\int_0^{n+p} r^2_{c}(j+n+p-s-p) \ \d s -(j+n)r^2_{c}(j+n)\Big)A^{-1}(c)\begin{bmatrix}
        (\partial^2_{\xi\xi} r^2_{c})(j-p)\\
        (\partial^2_{\xi c}r^2_{c})(j-p)
    \end{bmatrix}\\
    =&- \frac{\alpha_0^{-1}(c)}{4c^2}\sum_{j\in\Z}\Big(\int_0^{n+p} r^2_{c}(j-p+\tau) \ \d \tau -(j+n)r^2_{c}(j+n)\Big)A^{-1}(c)\begin{bmatrix}
        (\partial^2_{\xi\xi} r^2_{c})(j-p)\\
        (\partial^2_{\xi c}r^2_{c})(j-p)
    \end{bmatrix}.
    \end{align*}
    Hence, we identify
    \begin{align*}M^{1,0}_{\RomanIII,\mathrm{per}}(\xi,c)
    =& - \frac{\alpha_0^{-1}(c)}{4c^2}\sum_{j\in\Z}\int_0^{\infty} r^2_{c}(j-p+\tau) \ \d \tau A^{-1}(c)\begin{bmatrix}
        (\partial^2_{\xi\xi} r^2_{c})(j-p)\\
        (\partial^2_{\xi c}r^2_{c})(j-p)
    \end{bmatrix},\end{align*}
    and
    \begin{align*}
     M^{1,0}_{\RomanIII,\mathrm{trans}}(\xi,c)
    =&  \frac{\alpha_0^{-1}(c)}{4c^2}\sum_{j\in\Z}\Big(\int_{n+p}^\infty r^2_{c}(j-p+\tau) \ \d \tau +(j+n)r^2_{c}(j+n)\Big)A^{-1}(c)\begin{bmatrix}
        (\partial^2_{\xi\xi} r^2_{c})(j-p)\\
        (\partial^2_{\xi c}r^2_{c})(j-p)
    \end{bmatrix},
\end{align*}
which sum to $M_{\RomanIII}^{1,0}(\xi,c)$. The component $M^{1,0}_{\RomanIII,\mathrm{per}}(\xi,c)$ is clearly periodic with period 1. It remains to be shown that $ M^{1,0}_{\RomanIII,\mathrm{trans}}(\xi,c)$ is exponentially decaying in $|\xi|=|n+p|$. The exponential decay of the wave profiles and their derivatives \cite[Proposition 5.5]{FPI} provides
\[ \sum_{j\in\Z} (j+n)r^2_{c}(j+n)\left|A^{-1}(c)\begin{bmatrix}
        (\partial^2_{\xi\xi} r^2_{c})(j-p)\\
        (\partial^2_{\xi c}r^2_{c})(j-p)
    \end{bmatrix}\right|\leq C \sum_{j\in\Z}e^{-\beta|j+n|}e^{-\beta|j-p|}\leq\tilde{C}e^{-\tilde{\beta}|n|},\]
for some $C,\tilde{C},\beta>0$. For the integral component, we similarly find
\begin{align*}
    \sum_{j\in\Z}\int_0^{n+p} r^2_{c}(j-p+\tau) \ \d \tau\left|A^{-1}(c)\begin{bmatrix}
        (\partial^2_{\xi\xi} r^2_{c})(j-p)\\
        (\partial^2_{\xi c}r^2_{c})(j-p)
    \end{bmatrix}\right| \leq  \sum_{j\in\Z}\int_0^{n+p} e^{-\beta|j-p+\tau|}e^{-\beta|j-p|} \ \d \tau \leq \tilde{C}e^{-\tilde{\beta}|n|},
\end{align*}  
upon increasing $\tilde{C}$ if necessary.
\end{proof}

\section{Numerical Schemes}\label{app:scheme}
Our direct simulations of the FPUT system~\eqref{eqn:FPU} were carried out using a fourth-order Runge--Kutta scheme. We extracted the modulation parameters from the numerical solution by enforcing the orthogonality conditions~\eqref{eqn:ortho} through a nonlinear solver. Specifically, given the numerical solution $u(t) = (r(t),\, p(t))^\top$ to~\eqref{eqn:FPU} at time $t$, we determined parameters $c_{\mathrm{fit}}(t)$ and $\xi_{\mathrm{fit}}(t)$ from the equations
\[
\Omega\Big(\partial_\xi \phi_{c_{\mathrm{fit}}(t)}(\cdot-\xi_{\mathrm{fit}}(t)),\, u-\phi_{c_{\mathrm{fit}}(t)}(\cdot-\xi_{\mathrm{fit}}(t))\Big)=0,
\]
and
\[
\Omega\Big(\partial_c \phi_{c_{\mathrm{fit}}(t)}(\cdot-\xi_{\mathrm{fit}}(t)),\, u-\phi_{c_{\mathrm{fit}}(t)}(\cdot-\xi_{\mathrm{fit}}(t))\Big)=0.
\]
We computed the wave profiles and their derivatives by solving the traveling wave equation~\eqref{eqn:TWE} following the approach outlined in \cite[Section 3.1]{nanopteron4}. The phase parameter $\gamma_{\mathrm{fit}}$ was then obtained via the numerical integration
\[
\gamma_{\mathrm{fit}}(t) \,=\, \xi_{\mathrm{fit}}(t)-\int_0^t c_{\mathrm{fit}}(s)\, \d s.
\]
For comparison, the theoretically equivalent parameters $c_{\mathrm{mod}}(t)$ and $\gamma_{\mathrm{mod}}(t)$ were computed by numerically integrating~\eqref{eqn:postl}--\eqref{eqn:etadot}, again with a fourth-order Runge--Kutta method. \Cref{fig:fitvsmod} demonstrates the close agreement between the ``fitted'' and ``modulation'' parameters, with discrepancies limited to negligible truncation errors. Throughout the numerical figures in this paper, we have used $c(t) = c_{\mathrm{fit}}(t)$ and $\gamma(t) = \gamma_{\mathrm{fit}}(t)$.

  \begin{figure}[h]
    \centering
    \begin{subfigure}[t]{0.48\textwidth}
        \centering
        \includegraphics[width=\linewidth]{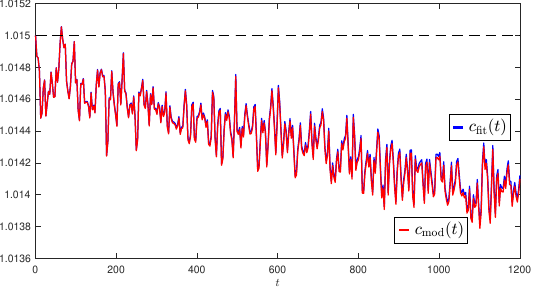}
    \end{subfigure}
    \hfill
    \begin{subfigure}[t]{0.48\textwidth}
        \centering
        \includegraphics[width=\linewidth]{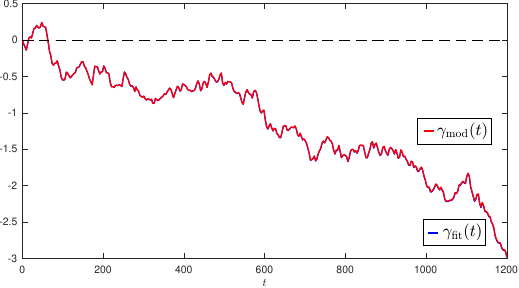}
    \end{subfigure}

    \caption{Comparison of $c_{\mathrm{fit}}(t)$ with the theoretically equivalent parameter $c_{\mathrm{mod}}(t)$ (left) and comparison of $\gamma_{\mathrm{fit}}(t)$ with $c_{\mathrm{mod}}(t)$ (right).}
    \label{fig:fitvsmod}
\end{figure}
For the numerical evaluation of the asymptotic tail $\overline{\eta}^\infty$ in \Cref{fig:tail}, based on $\overline{\eta}=\eta^{\mathrm{h}}_1$, we computed the corresponding asymptotic response function $\overline{R}^\infty$ directly from a discretization of~\eqref{eqn:Rinf} with the explicit kernel representation from \Cref{lem:discwave}. Additionally, we computed the response function $\overline{R}^\infty$ based on $\overline{\eta}=\eta_1$ using the true soliton linearization $\mathcal{A}_{c_*t,c_*}=\mathcal{L}_{c_*t,c_*}$, performing a numerical integration of~\eqref{eqn:response} until satisfactory convergence was achieved. We then used the resulting $\overline{R}^\infty$ functions to determine the asymptotic attenuation rate $\mathcal{Q}_c(c_*)$ according to \Cref{cor:limit} and~\eqref{eqn:rate}, in which the periodic quantity $M^{1,0}_{\mathrm{per}}(p,c_*)$ was evaluated by computing $M^{1,0}(p+n,c_*)$ for sufficiently large $n$.
\bibliographystyle{plain}
\bibliography{references}
\end{document}